\documentclass{article}

\usepackage{amssymb}
\usepackage{amsmath}
\usepackage{amsthm}
\usepackage{geometry}
\usepackage[hidelinks]{hyperref}
\usepackage{natbib}
\usepackage{authblk}
\newtheorem{Theorem}{Theorem}
\newtheorem{Definition}{Definition}
\newtheorem{Lemma}{Lemma}
\newtheorem{Proposition}{Proposition}

\newcommand{\real}{\mathbb{R}}

\newcommand{\suppf}{\sup_{f \in \Sigma(\beta,L)}}
\newcommand{\Exp}{\mathbb{E}}
\newcommand{\betahold}{\Sigma(\beta,L)}
\newcommand{\fracbeta}[1]{\frac{#1}{2\beta+d}}
\newcommand{\lonenorm}[1]{\left\Vert{#1}\right\Vert_1}
\newcommand{\euclideannorm}[1]{\left\Vert{#1}\right\Vert_2}
\newcommand{\supnorm}[1]{\left\Vert{#1}\right\Vert_{\infty}}
\newcommand{\goodpar}[1]{\left(#1\right)}
\newcommand{\goodbrak}[1]{\left[#1\right]}
\newcommand{\randomev}{\mathcal{A}_{ji}}
\newcommand{\KL}{\text{KL}}
\newcommand{\roundbtx}{\mathcal{B}_{T}(x)}
\newcommand{\inderound}{\mathbf{I}\{\mathcal{E}\}}
\newcommand{\cardinality}[2]{D_{{#1},{#2}}}
\newcommand{\derivativealpha}{f^{(\alpha)}}
\newcommand{\exponentalpha}{^{(\alpha)}}
\newcommand{\estimftalpha}{\hat{f}^{(\alpha)}_T}
\newcommand{\estimfjalpha}{\hat{f}^{(\alpha)}_j}
\newcommand{\proba}{\mathbb{P}}

\newcommand{\minimumfalpha}{f^{(\alpha)}_*}

\newcommand{\weights}{W_{T,i}^{(\alpha)}(x)}

\DeclareMathOperator*{\argmin}{arg\,min}

\begin{document}
\linespread{1.}
\title{Minimax optimality for sequential gradient-free minimization of smooth functions and their derivatives} 

\author[1,2]{Théo Paquier}
\author[1]{Alexandre B. Tsybakov}
\author[2]{François Portier}
\author[2]{Mohammadreza M. Kalan}
\affil[1]{ENSAE, CREST, 5 avenue Le Chatelier, Palaiseau, 91120, France}
\affil[2]{Univ Rennes, Ensai, CNRS, CREST—UMR 9194, F-35000 Rennes, France}
\date{}
\maketitle

\begin{abstract}
We consider the problem of noisy gradient-free minimization of the $k$-th order partial derivative of a $\beta$-H\"older function supported on a $d$-dimensional cube. We show that $T^{\fracbeta{\beta+d+k}}\log(T)^{\fracbeta{\beta-k}}$ is a non-asymptotic minimax rate of the $T$ step cumulative regret for all $\beta>0$. In the special case $k=0$, our results cover the problem of noisy gradient-free minimization of $\beta$-H\"older functions, closing the existing gap between the known upper and lower bounds. We show that a minimizer of a suitably chosen local polynomial estimator is rate-optimal. The minimax optimal upper bound is achieved under the passive design, that is, when the query points are i.i.d. Thus, there is no advantage in considering sequential designs when it is only known that $f$ is a $\beta$-H\"older function with no additional property. We propose an algorithm feasible in polynomial time that constructs a proxy of the minimizer of the local polynomial estimator. The procedure requires computing the estimator on auxiliary random points. The resulting polynomial time algorithm matches the lower bound.  
\end{abstract}

\section{Introduction}

We consider the sequential setup such that at each time instance $i\in\{1,\dots,T\}$, the learner chooses a query point $X_i\in[0,1]^d$ and observes the response 
\begin{equation}\label{model}
  Y_i = f(X_i)+\xi_i,  
\end{equation}
where $f:[0,1]^d\to \mathbb{R}$,  $\xi_i$ is a noise variable, $X_i = \Phi_i((Y_t,X_t)_{t=1}^{i-1})$, and $\Phi_i$ is a measurable function (this function can include as an extra argument a randomization variable generated by the learner). The aim of the learner is to minimize a partial derivative of order $k$ of $f$. We will refer to this derivative as $f\exponentalpha$ where $\alpha\in\mathbb{N}^d$ is a multi-index such that $\sum_{j=1}^d \alpha_j = k$. At each time instance $i$, the learner outputs an estimator $\hat{z}_i$, which is a measurable function of the observations $(X_t,Y_t)_{t=1}^i$. The accuracy of the estimation procedure is measured either by the simple regret, defined as $\Exp_f[f\exponentalpha(\hat{z}_T)]-\min_{x\in[0,1]^d}f\exponentalpha(x)$, or by the cumulative regret 
\[
R_T(f) := \sum_{i=1}^T \goodpar{\Exp_f{\goodbrak{f\exponentalpha(\hat{z}_i)}}-\min_{x\in[0,1]^d}f\exponentalpha(x)}.
\]
Here, $\Exp_f[\cdot]$ denotes the expectation with respect to the joint distribution of $((X_i,Y_i), i=1,\dots,T)$ satisfying \eqref{model}. The objective of the learner is to construct an estimator with small simple regret or a sequence of estimators with small cumulative regret.

For $k=0$, that is, when the minimization of the function $f$ itself is considered, this setting or similar ones have been extensively studied in the literature. The main stream of this line of work, known as gradient-free (or zero-order) stochastic optimization, focused on analyzing properties of the estimators under smoothness assumptions on $f$ combined with structural assumptions such as convexity, strong convexity, Polyak-\L ojasiewicz condition and others; see \cite{kiefer1952stochastic,fabian1967,polyak1990optimal,tsybakov1990passive,dippon2003accelerated, pelletier07,agarwal2010optimal,ghadimi2013stochastic,shamir2013complexity, duchi2015optimal,
bach2016highly,nesterov2017random,shamir2017optimal,locatelli2018adaptivity,akhavan2020exploiting,akhavan2024gradient,akhavan2024estimating,fokkema2024online,lattimore2026bandit}
and the references cited therein. Some of these papers treat a more pointed setting of continuum bandit optimization. In contrast to the general gradient-free optimization, the bandit setting imposes an extra requirement that the estimators $\hat{z}_i$ should coincide with $X_i$. Thus, the bandit solutions also solve the gradient-free optimization problem but not vice versa.  

In the present paper, we consider the setting where only smoothness assumptions are imposed on $f$, namely, we assume that $f$ is $\beta$-H\"older for a given $\beta>0$ (see Definition \ref{def:holder_class_chhor} below).  The prior work on this setting considered only the case $k=0$. In dimension $d = 1$, the early papers by \cite{kleinberg2004} and \cite{auer2007improved} provided bandit algorithms with cumulative regret of the order $O(T^{(\beta+1)/(2\beta+1)}\log(T)^{\beta/(2\beta+1)})$ for $0<\beta\le1$. 
Later, \cite{kleinberg2008,bubeck2011x} considered general dimension $d$ but only Lipschitz regularity classes, i.e., $\beta =1$, and required additional conditions. In \cite{bubeck2011x}, the function is assumed to be well-behaved with respect to a dissimilarity function near the optima of the mean payoff, and the distance between any two maxima is controlled by the same dissimilarity. Hence, the comparison to the class of $\beta$-H\"older functions is not straightforward. Under different assumptions, these two papers proposed bandit algorithms with cumulative regret scaling as $  T^{( 1+d) / (2 +d)} $ up to logarithmic factors.  The algorithm in \cite{kleinberg2008} requires knowledge of a topological oracle, whereas \cite{bubeck2011x} designs a directly feasible polynomial time procedure.
In a further development, \cite{Wang-Balakrishnan-Singh2019} derived upper and lower bounds for the simple regret for any $\beta>0$ and any dimension $d$ under the assumption of i.i.d. Gaussian noise $\xi_i$. Their lower bound is of the order $T^{-\beta/(2\beta+d)}$, which is smaller than the  optimal rate $T^{-\beta/(2\beta+d)}\log(T)^{\beta/(2\beta+d)}$ obtained as a special case of our results\footnote{\cite{Wang-Balakrishnan-Singh2019} deals with simple regret, so to compare with the rate for cumulative regret one needs to multiply their rates by $T$.}.  The upper bound in \cite{Wang-Balakrishnan-Singh2019} carries an additional factor $\log^{a}(T)$ with an unspecified $a>1$ that can be large. Using a different approach, 
\cite{liu2021smooth} proposed a bandit algorithm with cumulative regret of the order $T^{(\beta+d)/(2\beta+d)}\log(T)^{(6\beta + d)/(2(2\beta+d))}$ for any $\beta>0$ and any dimension $d$.   
Another paper that tackled the cumulative regret with similar assumptions is \cite{singh2021continuum}, claiming the minimax rates for the case $k = 0$ on Besov classes rather than on H\"older classes, without a valid proof however\footnote{The upper bound on the regret in \cite{singh2021continuum} is claimed for smoothness parameter and dimension in general position but its proof is based on referring to \cite{auer2007improved} that only covers the $\beta$-Hölder case with $0<\beta\le1$, and $d=1$.}. In summary, the prior work focuses only on the setting with $k=0$ and provides the minimax rates to within logarithmic factors but does not specify the exact minimax rate or an algorithm achieving it.  

The contributions of the present paper can be summarized as follows.
\begin{itemize}
    \item We prove that $T^{\fracbeta{\beta+d+k}}\log(T)^{\fracbeta{\beta-k}}$ is a non-asymptotic minimax optimal rate for the cumulative regret on the class of $\beta$-Hölder functions with any $\beta>0$, $d\ge 1$, $0\le k < \beta$. Our results apply as a special case to the problem of minimizing $f$ ($k=0$), closing the existing gap between the known upper and lower bounds for $k=0$.
    \item Our rate-optimal upper bound on the cumulative regret is proved under the passive design, that is, when $X_i$'s are i.i.d. Namely, we prove that choosing $X_1,\dots,X_T$ to be uniformly distributed on $[0,1]^d$ and then using a minimizer of a local polynomial nonparametric regression estimator is enough to achieve the minimax rate. Thus, we show that there is no advantage in considering sequential strategies of choosing $X_i$'s when it is only known that $f$ is a $\beta$-Hölder function with no additional property.
    \item We propose an estimator, which is feasible in polynomial time and achieves the minimax optimal rate. It is obtained by minimizing the local polynomial estimator on a carefully chosen auxiliary random grid.  
    \end{itemize}
   
This paper is organized as follows. Section \ref{sec:def_not} introduces the main definitions and notation. In Section \ref{sec:upper_bound}, we state the upper bound and develop more thoroughly on the use of local polynomial estimator (LPE). We prove a non-asymptotic bound for the sup-norm risk of the LPE in dimension $d$ with random uniform design. In Section \ref{sec:poly_time_estim}, we propose a polynomial time estimator that attains the optimal rate. Section \ref{sec:lower_bound} is devoted to the lower bound. 

\section{Definitions and notation}\label{sec:def_not}

To state the definitions, we will consider a function $g : [0,1]^d \rightarrow \real$, with $d \geq 1$. As multivariate polynomials and derivatives will be used, we introduce notation adapted to dimensions $d\geq 1$. We focus on convergence rates up to numerical constants and do not aim to derive sharp constants. Therefore, throughout the following $C, C',C'',C_1,C_2,\dots$ will denote positive constants that can depend on the ambient dimension and the regularity of the function but not on $T$. The value of those constants can change from line to line. Let $\left\Vert \cdot \right\Vert_2$ and $\lonenorm{\cdot}$ denote the Euclidean norm and the $l_1$ norm in $\real^d$, respectively. The Euclidean ball of radius 1 centered at $0$ in dimension $d$ is denoted by $\mathbb{B}_d$. We denote by $\mathbf{I}\{A\}$ the indicator function of the set $A$, and by $\lambda(A)$ its Lebesgue measure. We denote by $\mathcal{U}([0,1]^d)$ the uniform distribution on $[0,1]^d$. 

We use the following definition of conditionally sub-Gaussian (SG) random variable.
\begin{Definition}[Conditional $\sigma-SG$ random variable]\label{def:cond_sg}
Let $\eta$ be a real-valued random variable and let $X$ be a random variable with values in $\mathbb{R}^m$, let $\sigma>0$. A random variable $\eta$ is $\sigma-SG$ conditionally on $X$ if :
\begin{equation*}
    \forall t \in \real,\; \Exp\goodbrak{\exp(t\eta)\vert X} \leq \exp\goodpar{\frac{\sigma^2t^2}{2}}.
\end{equation*}
This implies that $\Exp[\eta\vert X]=0$. We can extend this definition to random vectors in $\real^d$. We say that a random vector $\eta\in\real^d$ is $\sigma-SG$ conditionally on $X$ if all its projections $\eta^{\top}v$ with $\euclideannorm{v}=1$ are $\sigma-SG$ random variables conditionally on $X$.
\end{Definition}

For any multi-index $s \in \mathbb{N}^d, \; s = (s_1,\dots,s_d)$, and any vector $x\in\mathbb{R}^d$ with $x=(x_1,\dots,x_d)$ set :
\[
     x^{s} = \prod_{i=1}^d x_{i}^{s_i}, \quad s! = \prod_{i=1}^d s_i !, \quad D^s = \frac{\partial ^{s_1+\dots+s_d}}{\partial x_1^{s_1}\dots\partial x_{d}^{s_d}}.
\]
For $x=(x_1,\dots,x_d)\in[0,1]^d$, we define the sup norm $\supnorm{\cdot}$ for functions and vectors as $\supnorm{g} = \sup_{x\in [0,1]^d}\vert g(x)\vert$ and $\supnorm{x}=\sup_{i\in\{1,\dots,d\}}|x_i|$. For any compact set $\Theta$, we denote its boundary by~$\partial\Theta$. We denote by $g^{(s)}$ the $s$-th derivative of $g$, i.e., for any multi-index $s$,  
\[
g^{(s)}(x) = D^{s}g(x) = \frac{\partial ^{s_1+\dots+s_d}}{\partial x_1^{s_1}\dots\partial x_{d}^{s_d}}g(x).
\]
There exist different definitions of H\"older classes. We will use the definition of \cite{chhor2024benign}, which is slightly more restrictive than the one given in \cite{stone1982optimal}. 

For any $x\in\real$ we denote by $\lfloor x\rfloor$  the largest integer strictly smaller than $x$.

\begin{Definition}[$\beta$-H\"older functions]\label{def:holder_class_chhor}
    Let $\beta>0$, $L>0$, and let $\ell=\lfloor \beta \rfloor$. We denote by $\Sigma(\beta,L)$ the set of all functions $g:[0,1]^d\to\mathbb{R}$ that are $\ell$ times continuously differentiable and satisfy the condition 
    \begin{equation}\label{eq:def_holder_chhor}
        \max_{\lonenorm{s'}\leq \ell}\sup_{x\in[0,1]^d} \vert g^{(s')}(x)\vert + \max_{\lonenorm{s}= \ell}\sup_{x,x'\in[0,1]^d} \frac{\vert g^{(s)}(x) - g^{(s)}(x')\vert}{\euclideannorm{x-x'}^{\beta-\ell}} \leq L.
    \end{equation}
\end{Definition}
We consider the class of sequential strategies of choosing the query points defined as follows. The learner chooses the query point as $X_i=\Phi_i((X_t,Y_t)_{t=1}^{i-1},\tau)$ for $i=2,\dots,T,$ where $\Phi_i$ is a measurable function, $X_1$ is any random variable, and $\tau$ is a random variable with values in a measurable space ($\mathcal{Z},\mathcal{U}$).  The variable $\tau$ is chosen by the learner and it represents a possible randomization. We denote by $\Pi_T$ the set of all such strategies of choosing query points $X_i$ for $i=1,\dots,T$. Elements of $\Pi_T$ are joint distributions of $(X_1,\dots,X_T)$. For such strategies, at each step $i$ the learner chooses a query point $X_i$, gets a feedback $Y_i$ given by \eqref{model}, where $f : [0,1]^d \rightarrow\real$ is a $\beta$-H\"older function, and outputs an estimator $\hat{z}_i = \Psi_i((Y_t,X_t)_{t=1}^{i},\tau)$. 

Let the multi-index $\alpha\in\mathbb{N}^d$ and the order $k\in\mathbb{N}$ of derivation be fixed such that $\lonenorm{\alpha}=k\leq\ell$. Let $x^*\in\argmin_{x\in[0,1]^d}f\exponentalpha(x)$ and $\minimumfalpha = f\exponentalpha(x^*)$. The minimax cumulative regret and the minimax simple regret are defined, respectively, as  
\begin{align}\label{eq:def_minimax_cumul_regret}
    R_T^* &= \inf_{\substack{(\hat{z}_i)_{i=1}^T, \\ \Pi_T}} \ \suppf  \sum_{i=1}^T\Exp_f\goodbrak{ \derivativealpha(\hat{z}_i) - \minimumfalpha},
    \\ \label{eq:def_minimax_simple_regret}
    R_T^S &= \inf_{{\hat{z}_T, \Pi_T}} \ \suppf  \Exp_f[f\exponentalpha(\hat{z}_T)- \minimumfalpha],
\end{align}
where the infimum is taken over all sequences of estimators $(\hat{z}_i)_{i=1}^T$ (respectively, all estimators $\hat{z}_T$) and all sequential strategies in~$\Pi_T$.

\section{Upper bound}\label{sec:upper_bound}

We use the following observation to control the regret. Let $\hat{S}_i$ be any estimator of the function $\derivativealpha$. Denote one of its minimizers by $\hat{x}_i$. Then it is straightforward to see that $\derivativealpha(\hat{x}_i)-\derivativealpha(x^*)\leq 2\Vert \hat{S}_i-\derivativealpha\Vert_{\infty}$. Applying this inequality we get that, for any estimator $\hat{S}_i$ with minimizer $\hat{x}_i$, the following holds:
\begin{equation}\label{ineq:general_sup_norm_bound}
R_T(f) = \sum_{i=1}^T \Exp_f\goodbrak{f\exponentalpha(\hat{x}_i)-f\exponentalpha_*} \leq 2\sum_{i=1}^T \Exp_f{\supnorm{\hat{S}_i - f\exponentalpha}}.
\end{equation}
Thus, a non-asymptotic bound on the right hand side summands allows one to control the cumulative regret for any $T$. 
We apply this argument by choosing $\hat{S}_i$'s as local polynomial estimators constructed with passive design, that is, when the query points $X_1,\dots,X_T$ are taken as i.i.d. random vectors.

\subsection{Upper bound in sup-norm error of the local polynomial estimator}\label{sec:proof_upp_lpe}

The local polynomial estimator (LPE) approximates a function's Taylor expansion. This is suitable to estimate derivatives. We consider the integer $\ell\ge 0$ that defines the order of the LPE. Dealing with $\beta$-H\"older functions, the natural choice is $\ell= \lfloor \beta\rfloor$. Let $\cardinality{\ell}{d} = \binom{d+\ell}{d}$ be the cardinality of the set of multi-indices $s=(s_1,\dots,s_d)$ such that $0\leq \lonenorm{s} \leq \ell$. We order $s^{(1)},\dots, s^{(\cardinality{\ell}{d})}$ with increasing values of their $l_1$ norm and ties can be broken arbitrarily. We define the vector valued function $U : \real^d \rightarrow \real^{\cardinality{\ell}{d}}$ as :
\begin{equation*}
\forall \; x \in \real^d, \quad  U(x) := \left(\frac{x^{s}}{s!}\right)_{\lonenorm{s} \leq \ell}.
\end{equation*}
For a multi-index $s$ such that $\lonenorm{s}\leq \ell $ and for any $x=(x_1,\dots,x_d)$ in $\real^d$, we have
\begin{equation}\label{eq:def_deriv_u}
    D^{\alpha}\frac{x^s}{s!} = \frac{1}{s!}D^{\alpha}\prod_{i=1}^d x_i^{s_i} \notag 
    = \begin{cases}
        \displaystyle \frac{1}{s!}\prod_{i=1}^d\frac{s_i !}{(s_i-\alpha_i)!} x_i^{s_i - \alpha_i} = \frac{1}{(s-\alpha)!}x^{s-\alpha} \quad \text{if } \; s_i \geq \alpha_i, \\
        0 \quad \text{otherwise}.
    \end{cases}
\end{equation}
We denote by $U\exponentalpha(\cdot)$ the component-wise $\alpha$-derivative of $U(\cdot)$:
\begin{equation}\label{eq:def_u_deriv_alpha}
    \forall \; x \in \real^d, \quad U\exponentalpha(x) := \goodpar{D^{\alpha}\frac{x^{s}}{s!}}_{\lonenorm{s}\leq \ell}.
\end{equation}
For $x\in[0,1]^d$, a kernel $K:\real^d\to\real$ and a bandwidth $h>0$ let the matrix $\roundbtx$ be defined as :
\begin{equation}\label{eq:def_roundbtx}
\roundbtx := \frac{1}{Th^d}\sum_{i=1}^T U\goodpar{\frac{X_i - x}{h}}U\goodpar{\frac{X_i - x}{h}}^{\top}K\goodpar{\frac{X_i - x}{h}} \quad \in \mathbb{R}^{\cardinality{\ell}{d} \times \cardinality{\ell}{d}}.
\end{equation}
The LPE($\ell$) of $f\exponentalpha(x)$ is defined as : 
\begin{align}
    \tilde{f}\exponentalpha_T(x) & := \hat{\theta}^{\top}_T(x)U^{(\alpha)}(0)h^{-k} , \notag \\
    \text{with} \quad \hat{\theta}_T(x) &\in \arg \min_{\theta\in\mathbb{R}^{D_{\ell,{d}}}} \sum_{i=1}^T\left( Y_i - \theta^{\top} U\left(\frac{X_i - x}{h}\right)\right)^2 K\left(\frac{X_i - x}{h}\right), \label{eq:min_pb_for_lpe_def}
\end{align}
where  $U\exponentalpha(\cdot)$ is defined for the order $\ell$ in (\ref{eq:def_u_deriv_alpha}). If matrix $\roundbtx$ is invertible the LPE is uniquely defined, and it is a linear estimator since in this case it is straightforward to represent it as follows.

\begin{Lemma}\label{lem:def_LPE}
    When $\roundbtx$ is invertible, the LPE($\ell$) is equivalently defined as 
    \begin{equation}\label{def:estimfalpha_w_weights}
        \tilde{f}\exponentalpha_T(x) = \sum_{i=1}^T Y_i \weights,
    \end{equation}
    where, for $i= 1,\dots,T$,  
    \begin{equation*}
        \weights = \frac{U^{(\alpha)}(0)^{\top}}{Th^{d+k}}\roundbtx^{-1}U\goodpar{\frac{X_i - x}{h}}K\goodpar{\frac{X_i - x}{h}}.   \end{equation*}
\end{Lemma}
 If $\roundbtx$ is not invertible the minimization problem in (\ref{eq:min_pb_for_lpe_def}) does not have a unique solution. Representation analogous to \eqref{def:estimfalpha_w_weights} can still be used in this case. It suffices to replace $\roundbtx^{-1}$ by the Moore-Penrose pseudo-inverse to get an estimator that has properties similar to the LPE with invertible matrix. Under mild conditions,  the probability that $\roundbtx$ is singular decays exponentially with the sample size, see Lemma~\ref{lem:lemma_proba_eigen} below. Thus, on an event of overwhelming probability, $\roundbtx$ is invertible and one can consider $\roundbtx^{-1}$ instead of its pseudo-inverse. 

We now present the assumptions used below to prove the upper bound on the regret under passive design.

\begin{description}
        \item[Assumption 1] $(X_i,\xi_i)_{i=1}^T$ is an independent sequence such that $X_i \sim \mathcal{U}\goodpar{[0,1]^d}$ for each $i=1,\dots, T$.
        \item[Assumption 2] For any $i\in\{1,\dots,T\}, \;\xi_i$ is $\sigma_{\xi}$-sub-Gaussian conditionally to $X_i$, with $0<\sigma_{\xi}<\infty.$
        \item[Assumption 3] The kernel $K$: $\mathbb{R}^d\rightarrow \mathbb{R}$ is Lipschitz continuous, i.e., for all $(u,v)\in \mathbb R^d\times\real^d$, $| K(u) - K(v) | \leq L_K \|u-v\|_2 $ with  $L_K>0$, and has a compact support. Moreover, there exist constants $\Delta>0, \; c>0$ and $K_{\max}<\infty$ such that $K_{\max} > K(u) \geq c \mathbf{I}\{\euclideannorm{u}\leq \Delta\}$.
    \end{description}
In the sequel, we consider the modified version of LPE($\ell$) defined for any $x$ in $[0,1]^d$ as follows:
\begin{equation}\label{def:hatLP}
    \hat{f}_T\exponentalpha(x) =
       \min(L, \max(-L,\tilde{f}\exponentalpha_T(x))).
\end{equation}
This modification introduces projection of the estimator values on the set $[-L,L]$ that only reduces the estimation error if $f\in\betahold$.  We can now state the following proposition.

\begin{Proposition}[Non-asymptotic $L_{\infty}$ risk of the LPE]\label{prop:prop_upper_bound_lpe} Let $T\ge 2$, $\beta >0$, $\alpha \in \mathbb N^d$ be such that $k= \lonenorm{\alpha} \leq \ell=\lfloor\beta\rfloor$. Let Assumptions 1 -- 3 hold. Let $\estimftalpha$ be the local polynomial estimator \eqref{def:hatLP} of order $\ell$ with bandwidth 
    \[
    h = a \goodpar{\frac{\log(T)}{T}}^{\fracbeta{1}},
    \]
    where $a>0$. Then there exists a constant $C^+$ such that 
    \[
    \sup_{f\in\betahold}\Exp_f\goodbrak{\psi_T^{-2}\supnorm{\estimftalpha - \derivativealpha}^2} \leq C^+,
    \]
    where 
    \begin{equation*}
        \psi_T = \goodpar{\frac{\log(T)}{T}}^{\fracbeta{\beta-k}}.
    \end{equation*}
\end{Proposition}
The fact that the sup-norm error of the (non-modified) LPE($\ell$) $\tilde{f}_T$ achieves the rate $\psi_T$ is proved in a different form in \cite{stone1982optimal}. However, the result of \cite{stone1982optimal} is not applicable to bounding the cumulative regret in our setting since it is asymptotic in $T$ and it is derived in probability in a weak way that does not imply any control of the expectation.

\subsection{Upper bound for the regret}\label{subsec:upper_bound_cumu_regret}
Combining Proposition \ref{prop:prop_upper_bound_lpe} with (\ref{ineq:general_sup_norm_bound}) we get the following theorem.
\begin{Theorem}[Regret upper bound]\label{th:cumu_regret_up_bound}
Let the assumptions of Proposition \ref{prop:prop_upper_bound_lpe} hold. For $i=1,\dots,T$, let $\hat{x}_i$ be a minimizer over $[0,1]^d$ of the local polynomial estimator $\hat{f}\exponentalpha_i$. Then there exists a constant $C>0$ such that 
\begin{equation*}
    \sup_{f\in\betahold}\Exp_f\goodbrak{f\exponentalpha(\hat{x}_T) - \min_{x\in[0,1]^d}f\exponentalpha(x)} \leq C \goodpar{\frac{\log(T)}{T}}^{\fracbeta{\beta-k}},
\end{equation*}
\begin{equation*}
    \sup_{f\in\betahold}\Exp_f\goodbrak{\sum_{i=1}^T f\exponentalpha(\hat{x}_i) - \min_{x\in[0,1]^d}f\exponentalpha(x)} \leq C T^{\fracbeta{\beta+d+k}}\log(T)^{\fracbeta{\beta-k}}.
\end{equation*}
\end{Theorem}
The proof of Theorem \ref{th:cumu_regret_up_bound} is given in \ref{app:A}. 

Two remarks about Theorem \ref{th:cumu_regret_up_bound} are in order. First, the query points $X_1,\dots,X_T$ that we choose are i.i.d. (passive design). As we show in Section \ref{sec:lower_bound}, the rate given in Theorem \ref{th:cumu_regret_up_bound} cannot be improved by any sequential strategy in $\Pi_T$. Thus, it turns out that using passive design is sufficient to achieve the optimal performance in the class of all sequential strategies when it is only known that $f\in\betahold$. Second remark, the estimators $\hat{x}_i$ considered in Theorem \ref{th:cumu_regret_up_bound} are, in general, not computable in polynomial time. In the next section, we show that one can modify the procedure to construct a polynomial time estimator achieving the same rate.

\section{Polynomial time algorithm}\label{sec:poly_time_estim}

While we have established that the minimizer of a properly defined LPE($\ell$) achieves the rate given in Theorem~\ref{th:cumu_regret_up_bound}, finding the global minimum over $[0,1]^d$  is, in general, intractable in practice. In this section, we propose a procedure, which is realizable in polynomial time and achieves the same rate. The method is based on a randomized discretization of the considered LPE.

At first sight, a natural discretization scheme would be to consider minimization of the LPE on the set of i.i.d. sample points $\{X_1,\dots,X_T\}$ rather than on the whole $[0,1]^d$. This type of discretization is comparable to what was proposed in several works on density mode estimation \citep{abraham2004asymptotic, dasgupta2014optimal}. In those papers, the estimator of the mode is obtained by maximizing a nonparametric density estimator over the sample points. The same technique can be adopted in the kNN nonparametric regression setting \citep{jiang2019non}. However, in our framework this approach is not satisfying. While it does not fail completely, it does not lead to the rate of Theorem \ref{th:cumu_regret_up_bound} unless some restrictive conditions are imposed on $\beta$.  

In contrast, it turns out that minimization of $\estimfjalpha$ on points that can differ from i.i.d. query points $X_1,\dots,X_T$ leads to the desired result for any $\beta$. The number of discretization points required to achieve the rate of Theorem \ref{th:cumu_regret_up_bound} depends heavily on the smoothness parameter $\beta$. Under high regularity, a dense discretization set is necessary to match this rate, whereas under poor regularity sparser discretization suffices.

Let $V_1,\dots,V_T$ be integers, where $V_j$ is the size of the discretization set at step $j$ of the procedure. Consider the following protocol.
At time $j\in \{1,\dots,T\}$, the learner draws $V_j$ i.i.d. points $Z_1,\dots,Z_{V_j}$ from distribution $\mathcal{U}([0,1]^d)$, computes the values of $\estimfjalpha$ at these points and outputs the estimator 
\begin{align}\label{eq:def_of_feasible_estim}
    \hat{z}_{V_j} &\in \argmin_{1\leq i \leq V_j} \estimfjalpha(Z_i),
\end{align}
where $\estimfjalpha$ is the local polynomial estimator defined  in \eqref{def:hatLP} and based on $j$ observations from model (\ref{model}). We do not make any assumptions about independence between the discretization points $Z_1,\dots,Z_{V_j}$ and the query points $X_1,\dots, X_T$. In particular, we do not exclude that some of $Z_i$'s can coincide with some of $X_t$'s.  

For any $z_{V_j}^*$ satisfying  
$$   z^*_{V_j} \in \argmin_{1\leq i \leq V_j}f\exponentalpha(Z_i)$$
we have the following decomposition: 
\begin{align}\label{ineq:diff_for_estim_feasible_estim}
    f\exponentalpha(\hat{z}_{V_j})-\min_{x\in[0,1]^d}f\exponentalpha(x)&= \underbrace{f\exponentalpha(\hat{z}_{V_j}) - \estimfjalpha(\hat{z}_{V_j})}_{\leq\sup_{x}\vert \estimfjalpha(x)-f\exponentalpha(x)\vert} + \underbrace{\estimfjalpha(\hat{z}_{V_j}) - \estimfjalpha(z^*_{V_j})}_{\leq0} \notag \\
    & \qquad + \underbrace{\estimfjalpha(z^*_{V_j}) - f\exponentalpha(z^*_{V_j})}_{\leq\sup_x\vert\estimfjalpha(x)-f\exponentalpha(x)\vert}  + {f\exponentalpha(z_{V_j}^*) - \min_{x\in[0,1]^d}f\exponentalpha(x)} \notag \\ 
    &\leq 2\sup_{x\in[0,1]^d}\vert \estimfjalpha(x)-f\exponentalpha(x)\vert + f\exponentalpha(z_{V_j}^*) -\min_{x\in[0,1]^d}f\exponentalpha(x).
\end{align}
The first term in \eqref{ineq:diff_for_estim_feasible_estim} can be controlled via Proposition \ref{prop:prop_upper_bound_lpe}. To handle the second term $f\exponentalpha(z_{V_j}^*) -\min_{x\in[0,1]^d}f\exponentalpha(x)$, we use the following proposition. 
\begin{Proposition}\label{prop:rate_for_feasible_estim}
Let the assumptions of Proposition \ref{prop:prop_upper_bound_lpe} hold, and let $\estimfjalpha$ be the local polynomial estimator as in Proposition \ref{prop:prop_upper_bound_lpe}. Let $\hat{z}_{V_j}$ be the estimator defined in (\ref{eq:def_of_feasible_estim}) with the number $V_j$ of auxiliary points, where the learner {computes}  $\estimfjalpha$, chosen as $V_j=\lfloor j^{\gamma}\rfloor +1$. If $j\geq2$ and 
\[
\gamma > \fracbeta{d}\goodpar{\mathbf{I}\{k=\ell\}+(\beta-k)\mathbf{I}\{k<\ell\}},
\]
then there exists $C>0$ such that 
\[
\sup_{f\in\betahold}\Exp_f\goodbrak{f\exponentalpha(\hat{z}_{V_j})-\min_{x\in[0,1]^d}f\exponentalpha(x)}\leq C\goodpar{\frac{\log(j)}{j}}^{\fracbeta{\beta-k}}.
\]
\end{Proposition}
The proof of Proposition~\ref{prop:rate_for_feasible_estim} is given in \ref{app:C}. 

As in Section \ref{sec:upper_bound} (cf. the proof of Theorem \ref{th:cumu_regret_up_bound}) we can now derive from \eqref{ineq:diff_for_estim_feasible_estim}, Propositions~\ref{prop:prop_upper_bound_lpe} and~\ref{prop:rate_for_feasible_estim}  the following bound for the cumulative regret. 

\begin{Theorem}
    Let the assumptions of Proposition \ref{prop:rate_for_feasible_estim} hold, and $T\ge 2$. For $j=1,\dots, T$, let the $\hat{z}_{V_j}$ be the estimators defined in \eqref{eq:def_of_feasible_estim}. Then there exists a constant $C'>0$ such that
    \begin{equation*}
    \sup_{f\in\betahold}\Exp_f\goodbrak{f\exponentalpha(\hat{z}_{V_T}) - \min_{x\in[0,1]^d}f\exponentalpha(x)} \leq C' \goodpar{\frac{\log(T)}{T}}^{\fracbeta{\beta-k}},
\end{equation*}
    \[
    \sup_{f\in\betahold}\Exp_f\goodbrak{\sum_{j=1}^T f\exponentalpha(\hat{z}_{V_j}) - \min_{x\in[0,1]^d}f\exponentalpha(x)} \leq C' T^{\fracbeta{\beta+d+k}}\log(T)^{\fracbeta{\beta - k}}.
    \]
\end{Theorem}
Note that there are two regimes depending on the order of the derivative $k=\lonenorm{\alpha}$. The first corresponds to the case $k=\ell$. It can be interpreted as the low smoothness regime. Indeed, either the function itself lacks smoothness (this includes the estimation of $f$ when $\beta\le 1$), or the derivative that we are dealing with is the least regular one, as it is the highest order derivative for which only a regularity assumption with H\"older index less than or equal to $1$ is available. The second regime corresponds to $k<\ell$, which necessarily implies that $\beta\geq 1$. In this case, the target derivative is of lower order than the highest available derivative, and the function is sufficiently smooth. Consequently, in the low smoothness regime, the intrinsic difficulty of estimating the target quantity dominates, so increasing the number of points brings no additional improvement in the convergence rate. In particular, in this case one can use $\gamma=1$ implying that, at each step $i$, minimization of the LPE on the i.i.d. sample points $X_1,\dots,X_i$ preserves the rate of Theorem~\ref{th:cumu_regret_up_bound}. In contrast, in the higher smooth regime, achieving the minimax rate requires more sample points, as the additional regularity makes it possible to exploit finer local approximations of the function.

In view of \eqref{ineq:diff_for_estim_feasible_estim}, it is not necessary to generate the new $V_j$ discretization points at each step~$j$. It suffices to augment the $V_{j-1}$ i.i.d. points available at step $j-1$ by the newly generated $Z_{V_{j-1}+1},\dots,Z_{V_{j}}$.

Finally, we emphasize that in some cases (e.g., tuning hyper-parameters in machine learning models, computational engineering with simulator-based experiments), the main cost is in obtaining an additional response of the model, i.e., a pair $(X_{\text{new}},Y_{\text{new}})$. The procedure that is investigated here does not require getting any additional pairs but only evaluating the estimator $\estimfjalpha$ multiple times.

\section{Lower bound}\label{sec:lower_bound}

We now establish a minimax lower bound for the simple and cumulative regrets for any sequential strategies. To derive the lower bound, we assume that the noise distribution satisfies the following condition.

\begin{description}
    \item[Assumption 4] The noise variables $(\xi_i)_{i=1}^T$ are i.i.d. Moreover, for each $i$, the noise $\xi_i$ is independent of $X_i$ and of the past observations $(Y_t,X_t)_{t=1}^{i-1}$. The distribution $F_\xi$ of $\xi_i$ satisfies the condition
    \[
        \int \log \frac{dF_{\xi}(u)}{dF_{\xi}(u+v)} \, dF_{\xi}(u)
        \leq I_0 v^2,
        \quad \text{for all } |v|\leq v_0,
    \]
    for some constants $0<I_0<\infty$ and $v_0>0$.
\end{description}

This assumption is satisfied, for instance, by the Gaussian distribution.  

\begin{Theorem}[Lower bound for sequential strategies]\label{th:lower_bound}
Let Assumption 4 hold. Let $R_T^*$ and $R_T^S$ defined in \eqref{eq:def_minimax_cumul_regret} and \eqref{eq:def_minimax_simple_regret} be the minimax cumulative and simple regrets over the function class $\Sigma(\beta,L)$. Then for $T\ge5^{(2\beta/d)+1}$ there exists a constant $C>0$ such that
\[
R_T^* \geq C \,
T^{\frac{\beta+d+k}{2\beta+d}}
\log(T)^{\frac{\beta-k}{2\beta+d}},
\]
and 
\[
R_T^S \geq C \,
T^{-\frac{\beta-k}{2\beta+d}}
\log(T)^{\frac{\beta-k}{2\beta+d}}.
\]
\end{Theorem}
The proof of Theorem~\ref{th:lower_bound} relies on application of a Fano type argument (see, e.g.,~\cite{Tsybakov09}) to the present setting. The particularity of the problem lies in handling the regret based on the minimum of $f^{(\alpha)}$ rather than a classical estimation loss, and in controlling the Kullback-Leibler divergences under the dependence induced by the sequential strategies, where the query points depend on past observations. We refer to \ref{app:B} for the full argument.

Lower bounds under sequential strategies have been established in various other settings under the assumption that $k=0$ and the additional convexity or strong convexity assumptions, that is, when the optimal rates are much faster (\cite{polyak1990optimal,shamir2013complexity,duchi2015optimal,akhavan2020exploiting,akhavan2024gradient}). These papers mainly stated lower bounds for simple regret (implying those for cumulative regret via multiplication by $T$). For the bandit setting, the analysis is mostly focused on functions with H\"older index $\beta \leq 1$ (see \cite{auer2007improved,lattimore2026bandit} and the references therein).    
In a setting close to Theorem~\ref{th:lower_bound}, a lower bound for the case $k=0$ is presented in~\cite{Wang-Balakrishnan-Singh2019}, in a local version, and assuming that the noise is Gaussian and $L$ is large enough, see Theorem 2 of~\cite{Wang-Balakrishnan-Singh2019} dealing with the active design case. However, that bound is not sharp as it only establishes the lower rate $T^{-\frac{\beta}{2\beta+d}}$ for the simple regret rather than the optimal rate $T^{-\frac{\beta}{2\beta+d}}\log(T)^{\frac{\beta}{2\beta+d}}$.

An important implication of Theorem~\ref{th:lower_bound} is that active strategies do not provide a theoretical advantage for minimizing the simple or cumulative regret. Indeed, in Sections~\ref{sec:upper_bound} and~\ref{sec:poly_time_estim}, we exhibit passive procedures matching the same rate as in the lower bound. We also note that,  
although the procedure from Section~\ref{sec:poly_time_estim} has polynomial complexity, it can be computationally rather heavy for higher smoothness $\beta$. In particular, it can require more computational effort than certain active methods tailored for low smoothness $\beta\le 1$ (such as in \cite{bubeck2011x}).

\vspace{6mm}

{\bf Acknowledgements}

\vspace{2mm}

The work of Théo Paquier and Alexandre B. Tsybakov was supported by Labex Ecodec (ANR-11-LABEX-0047) and by ANR MaLIP (ANR-25-CE40-3228-01). 

\appendix

\section{Proof of the upper bound}\label{app:A}

{\it Proof of Theorem \ref{th:cumu_regret_up_bound}. }
Throughout the proofs below, we use for brevity the notation $\mathbb{E},\mathbb{P}$ instead of $\mathbb{E}_f,\mathbb{P}_f$, and the notation $D$ instead of $D_{\ell,d}$.

At time $i\in\{1\dots,T\}$, for any estimator $\hat{f}^{(\alpha)}_i$ of $\derivativealpha$ we have  $\derivativealpha(\hat{x}_i)-\derivativealpha(x^*)\leq 2\Vert \hat{f}\exponentalpha_i-\derivativealpha\Vert_{\infty}$, where $\hat{x}_i$ is a minimizer of $\hat{f}^{(\alpha)}_i$. 
By taking here $\hat{f}^{(\alpha)}_i$ as the local polynomial estimator under the assumptions of Proposition \ref{prop:prop_upper_bound_lpe}, and using the bound of Proposition \ref{prop:prop_upper_bound_lpe} we imediately get the results of the simple regret. Moreover, for the cumulative regret,  
\begin{equation}\label{ineq:appendix_upp_bound_first_cor}
\sum_{i=1}^T\Exp\goodbrak{{f}^{(\alpha)}(\hat{x}_i)-\minimumfalpha} \leq 2L+ 2C\sum_{i=2}^T \goodpar{\frac{\log(i)}{i}}^{\fracbeta{\beta-k}}\leq 2L+  2C\log(T)^{\fracbeta{\beta-k}} \mathtt{S},
\end{equation}
with 
\begin{equation}\label{ineq:upp_bound_second_bound_app_cor}
    \mathtt{S}= \sum_{i=1}^T \goodpar{\frac{1}{i}}^{\fracbeta{\beta-k}} \leq \int_1^T t^{-\fracbeta{\beta-k}}dt + 1 
    \leq \frac{2\beta+d}{\beta+d+k}T^{\fracbeta{\beta+d+k}},
\end{equation}
where the last inequality uses the fact that $k\leq\beta$. 
The bound of Theorem \ref{th:cumu_regret_up_bound} for the cumulative regret follows from  (\ref{ineq:appendix_upp_bound_first_cor}) and (\ref{ineq:upp_bound_second_bound_app_cor}).

{\it Proof of Proposition \ref{prop:prop_upper_bound_lpe}. } 
For the background on polynomial interpolation and properties of local polynomial estimators, we refer to \cite{wendland2004scattered} and \cite{Tsybakov09}. The proof is structured as follows. First, we control the smallest eigenvalue of $\roundbtx$. Then we provide a bound on the bias of the estimator. Finally, we deal with the variance of the estimator. 

{\it Control of the  smallest eigenvalue of $\roundbtx$. }
Let $\lambda_0>0$ be a constant. Introduce the following event
    \begin{equation*}
        \mathcal{E} := \{\inf_{x\in[0,1]^d} \lambda_{\min}(\roundbtx)\geq\lambda_0\}.
    \end{equation*}
We apply Lemma 4 from \cite{chhor2024benign} to control the probability of this event. It requires the following assumption ($\mathcal{H}$) : {\it The random vector $X$ is distributed with Lebesgue density $p(\cdot)$ such that $p \in [p_{\min}, p_{\max}]$ where $p_{\max}\geq p_{\min}>0$. The support Supp($p$) of $p$ is a convex compact subset of $\mathcal{B}_d$}. This assumption is satisfied for our $X_i$  up to a rescaling of the constants. Indeed, our $X_i$ are i.i.d. copies of a random variable $X\sim\mathcal{U}([0,1]^d)$ with support $[0,1]^d$ included in a Euclidean ball centered at 0 with radius greater than 1. The value of the radius only influences the constants involved in the lemma. 
\begin{Lemma}[Lemma 4 from \cite{chhor2024benign}]\label{lem:lemma_proba_eigen}
Let $h>0$. Let $X_1,\dots,X_T$ be i.i.d. copies of a random variable $X$ satisfying ($\mathcal{H}$), and let $K$ be a kernel satisfying Assumption 2 introduced in Section \ref{sec:upper_bound}.
Then there exist constants $\lambda_0>0, \; c >0$ depending only on $\ell, \; \Delta, \; d$  such that 
\begin{equation*}
    \mathbb{P}\goodpar{\inf_{x\in [0,1]^d}\lambda_{\min}(\roundbtx)\geq\lambda_{0}} \geq 1-c\goodpar{h^{-d^2-d}e^{-Th^d/c} + e^{-T^3h^{2d}/c}}
\end{equation*}
\end{Lemma}
With $h=a(\log(T)/T)^{1/(2\beta+d)}$, we obtain that there exists $c>0$ such that 
\begin{equation}\label{ineq:proba_rate_comp_inv_btx}
\mathbb{P}(\mathcal{E}^c) \leq ce^{-A_T /c},
\end{equation}
where $A_T = (T/\log(T))^{\fracbeta{2\beta}}$. The fact that $f\in\betahold$ and the definition of $\estimftalpha$ imply:   
\[
\forall \; x \in [0,1]^d, \; |\derivativealpha(x)| \leq L \quad \text{and}\quad  |\estimftalpha (x)|\leq L.
\]
These remarks and the fact that $\supnorm{\estimftalpha - \derivativealpha}\le \supnorm{ {\tilde f}^{(\alpha)}_T- \derivativealpha}$ imply: 
\begin{align}\label{ineq:ineq_first_decomp_sup_norm_bound_proba}
    \Exp\goodbrak{\supnorm{\estimftalpha - \derivativealpha}^2}
    &\le  \Exp\goodbrak{\supnorm{\estimftalpha - \derivativealpha}^2 \mathbf{I}\{\mathcal{E}\}} + (2L)^2\mathbb{P}(\mathcal{E}^c)\notag \\
    & \leq \Exp\goodbrak{\supnorm{ {\tilde f}^{(\alpha)}_T- \derivativealpha}^2 \mathbf{I}\{\mathcal{E}\}} + C \exp\goodpar{-\frac{A_T}{c}}.
\end{align}
Set $\tilde{\mathbb{E}}[\cdot]:=\Exp[\cdot|X_1,\dots,X_T]$ to simplify the notation. We proceed by decomposing the estimation error into bias and variance terms. Since $\forall \: a,b \in \real, \: (a+b)^2\leq 2a^2 + 2b^2$, we obtain:
\begin{equation}\label{bias-variance-decomp}
    \Exp\goodbrak{\supnorm{{\tilde f}^{(\alpha)}_T - \derivativealpha}^2 \inderound}\leq 2 \Exp \goodbrak{\supnorm{ {\tilde f}^{(\alpha)}_T - \tilde{\Exp}[{\tilde f}^{(\alpha)}_T] }^2\inderound} + 2 \Exp \goodbrak{\supnorm{ \tilde{\Exp}[{\tilde f}^{(\alpha)}_T] - \derivativealpha }^2\inderound}.
\end{equation}

{\it Control of the bias.}
We start by considering the squared bias term $\Exp \goodbrak{\supnorm{ \tilde{\Exp}[{\tilde f}^{(\alpha)}_T] - \derivativealpha }^2\inderound}$. On the event $\mathcal{E}$ the matrix $\roundbtx$ is invertible and  using Lemma \ref{lem:def_LPE} we have 
\[
\forall \; x \in [0,1]^d, \quad {\tilde f}^{(\alpha)}_T(x) = \sum_{i=1}^T Y_i \weights.
\]
Since the noise is conditionally sub-Gaussian, it is conditionally zero mean, so that 
\begin{equation}\label{eq:ub-0}
\tilde{\Exp}\goodbrak{{\tilde f}^{(\alpha)}_T(x)} = \tilde{\Exp}\goodbrak{\sum_{i=1}^T \goodpar{f(X_i) + \xi_i}\weights} = \sum_{i=1}^T f(X_i) \weights.
\end{equation}
Therefore,
\begin{equation}\label{ineq:bound_bias_from_above}
\tilde{\Exp}[{\tilde f}^{(\alpha)}_T(x)] - f\exponentalpha(x) = \sum_{i=1}^T f(X_i)\weights - f\exponentalpha(x).
\end{equation}
To pursue the proof, we need the following lemma.

\begin{Lemma}\label{lem:reproduction_poly}
Let $x\in\real^d$ be such that $\roundbtx \succ 0$. Let $Q$ be a polynomial of degree smaller than or equal to $\ell$. Let $\alpha=(\alpha_1,\dots,\alpha_d)$ be a multi-index such that $\lonenorm{\alpha}\leq \ell$. Then the weights $\weights, \; i \in \{1\dots,T\}$, are such that 
\begin{equation*}
    \sum_{i=1}^T Q(X_i)\weights = Q^{(\alpha)}(x).
\end{equation*}
    In particular, for every multi-index $s \in \mathbb{N}^d$ such that $\|s\|_1 \le \ell$, we have
    \begin{equation}\label{eq:lemma_reproduction_poly_part}
         \quad \sum_{i=1}^T (X_i - x)^s \weights = \alpha! \mathbf{I}\{s=\alpha\}.
    \end{equation}
\end{Lemma}

\begin{proof}
For clarity and legibility of the proof, we recall the definition of $U(\cdot)$ with multi-indices of maximal $l_1$-norm equal to $\ell$: 
\[
\forall \; x\in \real^d, \quad U(x):= \goodpar{\frac{x^s}{s!}}_{\lonenorm{s}\leq \ell}.
\]
The ordering of coordinates here does not matter, as long as it is kept the same throughout the computations. In fact, the only coordinate that has to remain in the same place is the coordinate of the multi-index $\alpha$. We consider the particular case $Y_i = Q(X_i)$ for all $i$, where $Q$ is a polynomial of degree smaller than or equal to $\ell$. To obtain the LPE($\ell$) solution we need to find the vector 
\begin{align}\label{eq:def_lpe_direct_obs_app}
    \hat{\theta}_T(x) &\in \argmin_{\theta\in\real^D} \sum_{i=1}^T\goodpar{Q(X_i) - \theta^{\top}U\goodpar{\frac{X_i - x}{h}}}^2 K\goodpar{\frac{X_i - x}{h}}.
\end{align}
Now, for any polynomial $Q$ of degree not exceeding $\ell$, a Taylor expansion yields 
\[
Q(X_i) = \sum_{\lonenorm{s}\leq \ell}\frac{Q^{(s)}(x)}{s!}(X_i - x)^s = q(x)^{\top}U\goodpar{\frac{X_i - x}{h}},
\]
where we set 
\[
q(x) := (Q(x), Q^{(s_1)}(x)h, \dots, Q^{(s_D)}(x)h^{\ell}), 
\]
and the multi-index $s_D$ is such that $\lonenorm{s_D}=\ell$. We can rewrite (\ref{eq:def_lpe_direct_obs_app}) as follows:
\[
    \hat{\theta}_T(x) \in \argmin_{\theta\in\real^D}\sum_{i=1}^T\goodpar{(q(x)-\theta)^{\top}U\goodpar{\frac{X_i - x}{h}}}^2 K\goodpar{\frac{X_i - x}{h}}
\]
or, equivalently,
\[
    \hat{\theta}_T(x) \in \argmin_{\theta\in\real^D} (q(x)-\theta)^{\top}\roundbtx(q(x)-\theta),
\]
where $\roundbtx$ is the matrix defined in (\ref{eq:def_roundbtx}). Since $\roundbtx$ is positive definite, we have that the unique solution is $\hat{\theta}_T(x) = q(x)$. By the definition of ${\rm LPE}(\ell)$,
\begin{equation*}
    {\tilde f}^{(\alpha)}_T(x) = \hat{\theta}_T(x)^{\top}U\exponentalpha(0)h^{-k}= q(x)^{\top}U\exponentalpha(0)h^{-k} = Q\exponentalpha(x).
\end{equation*}
Since we can write ${\tilde f}^{(\alpha)}_T(x)$ as in (\ref{def:estimfalpha_w_weights}) we have 
\begin{equation}\label{eq:reproduction_poly_deriv_alpha_appendix}
    \sum_{i=1}^T Q(X_i)\weights = Q\exponentalpha(x).
\end{equation}
To show (\ref{eq:lemma_reproduction_poly_part}) let $Q=Q_s$ be defined as : $\forall \;z\in\real^d, \;Q_s(z)=(z-x)^s$, for $s\in\mathbb{N}^d$. In this case,
\begin{equation}\label{eq:def_deriv_poly}
    Q_s\exponentalpha(z) =\begin{cases}
        \displaystyle(z-x)^{s-\alpha}\prod_{i=1}^d \frac{s_i!}{(s_i -\alpha_i)!}, \; \text{if} \; s_i\geq\alpha_i \; \forall i \in \{1,\dots,d\} \\
        0 \quad \text{otherwise}.
    \end{cases}
\end{equation}
In particular, $Q_{\alpha}\exponentalpha(x) = \alpha!$. This and (\ref{eq:reproduction_poly_deriv_alpha_appendix}) yield~\eqref{eq:lemma_reproduction_poly_part}.
\end{proof}
For any $i\in\{1,\dots,T\}$, the Taylor expansion of $f$ of order $\ell$ around point $x$ yields 
\[
    f(X_i) = \sum_{\lonenorm{s}<\ell}\frac{f^{(s)}(x)}{s!}(X_i - x)^s + \sum_{\lonenorm{s} = \ell}\frac{f^{(s)}\goodpar{x+\tau_{i}(X_i - x)}}{s!}(X_i - x)^s,
\]
with some $\tau_{i} \in [0,1]$. 
Multiplying by $\weights$ and summing up over $i$ we obtain 
\begin{align*}
    \sum_{i=1}^T f(X_i)\weights &=\underbrace{\sum_{i=1}^T \sum_{\lonenorm{s}<\ell}\frac{f^{(s)}(x)}{s!}(X_i - x)^s \weights}_{=:\mathtt{S_1}}  \\
    &+ \underbrace{\sum_{i=1}^T\sum_{\lonenorm{s} = \ell}\frac{f^{(s)}\goodpar{x+\tau_{i}(X_i - x)}}{s!}(X_i - x)^s\weights}_{=:\mathtt{S_2}}.
\end{align*}
We consider separately two cases, first $\lonenorm{\alpha} < \ell$ and then $\lonenorm{\alpha}=\ell$. For $\lonenorm{\alpha} < \ell$, applying Lemma~\ref{lem:reproduction_poly} we get
\begin{equation*}
    \mathtt{S_1} = \sum_{\substack{\lonenorm{s}<\ell \\ s\neq \alpha}}\frac{f^{(s)}(x)}{s!}\underbrace{\sum_{i=1}^T (X_i - x)^s \weights}_{=0 \; \text{by (\ref{eq:lemma_reproduction_poly_part})} } + \frac{f^{(\alpha)}(x)}{\alpha!}\underbrace{\sum_{i=1}^T (X_i - x)^{\alpha}\weights}_{=\alpha! \; \text{by (\ref{eq:lemma_reproduction_poly_part})}} = f\exponentalpha(x).
\end{equation*}
Furthermore, by Lemma \ref{lem:reproduction_poly} we have $\sum_{i=1}^T \weights(X_i - x)^s =0$ if $\lonenorm{s}=\ell$ and $\lonenorm{\alpha} < \ell$, so that the following equalities hold:
\begin{align*}
    \mathtt{S_2} &= \sum_{\lonenorm{s}=\ell}\sum_{i=1}^T \frac{f^{(s)}(x+\tau_i(X_i - x))}{s!}(X_i - x)^s\weights - \sum_{\lonenorm{s}=\ell}\sum_{i=1}^T \frac{f^{(s)}(x)}{s!}(X_i - x)^s \weights \\
    &= \sum_{i=1}^T \weights\sum_{\lonenorm{s}=\ell}\frac{f^{(s)}(x+\tau_i(X_i - x)) - f^{(s)}(x)}{s!}(X_i - x)^s=: \mathtt{S_3}.
\end{align*}
Next note that, for all $f\in \Sigma(\beta,L)$ and all $c\in [0,1]$, $x,y\in [0,1]^d$ we have (cf., for example, inequality (18) in \cite{chhor2024benign}):
\begin{equation*}
\left| \sum_{\lonenorm{s}=\ell}\frac{1}{s!}\goodbrak{f^{(s)}(y+c(x-y))-f^{(s)}(y)}(x-y)^s\right| \leq C_{b} \euclideannorm{x-y}^{\beta},
\end{equation*}
where $C_{b}>0$ is a constant.
Combining this remark with the above arguments we get that for $\alpha$ such that $\lonenorm{\alpha}<\ell$ the following holds:
\begin{equation}\label{ineq:taylor_ineq_chhor}
    \left\vert \sum_{i=1}^Tf(X_i)\weights - f\exponentalpha(x)\right\vert = \vert \mathtt{S_3}\vert \leq C_{b} \sum_{i=1}^T \left\Vert X_i-x\right\Vert_2^{\beta}|\weights|.
\end{equation}
The argument is slightly different for $\lonenorm{\alpha}=\ell$. In this case, Lemma \ref{lem:reproduction_poly} yields $\mathtt{S_1}=0$ while 
\begin{align*}
    \sum_{\lonenorm{s}=\ell}\sum_{i=1}^T \frac{f^{(s)}(x)}{s!}(X_i - x)^s \weights =& \sum_{\substack{\lonenorm{s}=\ell\\ s \neq \alpha}}\sum_{i=1}^T \frac{f^{(s)}(x)}{s!}(X_i - x)^s \weights \\
    &+ \sum_{i=1}^T \frac{f^{(\alpha)}(x)}{\alpha!}(X_i - x)^{\alpha}\weights \\
    =& f\exponentalpha(x).
\end{align*}
Therefore, we have $ \left\vert \sum_{i=1}^Tf(X_i)\weights - f\exponentalpha(x)\right\vert=|\mathtt{S_2} - f\exponentalpha(x)|=|\mathtt{S_3}|$, so that we can again use (\ref{ineq:taylor_ineq_chhor}). Hence, in both cases for any $x\in[0,1]^d$ the expression in (\ref{ineq:bound_bias_from_above}) is bounded as follows:
\begin{equation}\label{eq:ub-1}
    \left\vert\tilde{\Exp}\goodbrak{{\tilde f}^{(\alpha)}_T(x)} - f\exponentalpha(x)\right\vert \leq C_{b}\sum_{i=1}^T  \euclideannorm{X_i-x}^{\beta} \vert \weights\vert.
\end{equation}
Recall now that on the event $\mathcal{E}$ we have $\left\Vert \roundbtx^{-1} v\right\Vert_2 \leq \euclideannorm{v} /\lambda_0$ for all $v \in \mathbb{R}^{\cardinality{\ell}{d}}$, where $\lambda_0$ is the constant from Lemma \ref{lem:lemma_proba_eigen}. In addition, notice that $\Vert U^{(\alpha)}(0)\Vert_2 = 1$ (as it is the vector containing a single $1$ in the coordinate associated to $\alpha$ and all other entries 0). It follows that, on the event~$\mathcal{E}$, 
\begin{align*}
    |\weights| &\leq \frac{1}{Th^{d+k}}\left\Vert \roundbtx^{-1}U\goodpar{\frac{X_i - x}{h}}K\goodpar{\frac{X_i-x}{h}}\right\Vert_2 \\
    &\leq \frac{1}{Th^{d+k} \lambda_0}\left\Vert U\goodpar{\frac{X_i-x}{h}}\right\Vert_2 K\goodpar{\frac{X_i-x}{h}} \\
    & \leq \frac{1}{Th^{d+k}\lambda_0}K\goodpar{\frac{X_i - x}{h}}\sqrt{\sum_{0\leq\lonenorm{s}\leq \ell} \goodpar{\frac{C_K}{(s!)^2}}} \quad (\text{since Supp($K$) is compact}) \\
    &\leq \frac{\sqrt{C_K\cardinality{\ell}{d}}}{Th^{d+k}\lambda_0}K\goodpar{\frac{X_i-x}{h}},
\end{align*}
where $C_K>0$ is a constant depending on the support of $K(\cdot)$ and the dimension.
By substituting this bound in \eqref{eq:ub-1} and using the fact that $K$ is compactly supported we find that
\begin{equation}\label{eq:ub-2}
\Exp\goodbrak{\supnorm{ \tilde{\Exp}[{\tilde f}^{(\alpha)}_T] - \derivativealpha }^2\inderound}
     \leq C h^{2(\beta-k)} \Exp \goodbrak{\sup_{x \in [0,1]^d} \left\vert\frac{1}{Th^d}\sum_{i=1}^T V_i(x) \right\vert^2},
\end{equation}
where $C>0$ is a constant, and we introduced the notation
$$
V_i(x):= K\goodpar{\frac{X_i-x}{h}}.
$$
We now prove that the expectation on the right hand side of \eqref{eq:ub-2} is bounded by a constant. We use an $\epsilon$-net argument. We cover $[0,1]^d$ by $M$ equal cubes with edge length $\epsilon=T^{-b}$, where $b>0$ will be chosen later. The centers $x_1,\dots,x_M$ of the cubes constitute an $\epsilon$-net. The cardinality $M$ of this net is at most $(\lfloor1/\epsilon\rfloor+1)^d$. Since $\epsilon<1$ we get $M\le (2/\epsilon)^d=(2T^b)^d$. We have 
\begin{align} \nonumber
\sup_{x \in [0,1]^d} \Big\vert\sum_{i=1}^T V_i(x) \Big\vert^2
     & \leq \left(\max_{1\leq j \leq M} \Big\vert \sum_{i=1}^T V_i(x_j) \Big\vert  + \sup_{x,x' : \Vert x-x'\Vert_{\infty} \leq \epsilon} \Big\vert \sum_{i=1}^T (V_i(x')- V_i(x))  \Big\vert\right)^2
     \\ \label{eq:ub-3}
     & \le 2 \max_{1\leq j \leq M} \Big\vert \sum_{i=1}^T V_i(x_j) \Big\vert^2 +
     2 L_K^2h^{-2}\epsilon^2 T^2,
\end{align} 
where we have used the fact that $V_i$'s satisfy the Lipschitz condition, cf. Assumption 3. Note that, for $i=1,\dots,T$, $j=1,\dots,M$,
\begin{equation}\label{eq:ub-4}
    \Exp[V_i(x_j)] \leq h^d \int K(u) du, \quad 
    \Exp\goodbrak{V_i(x_j)^2} \leq h^d \int K^2(u)du.
\end{equation}
Using the first inequality in \eqref{eq:ub-4} and the notation $\zeta_{i,j}=V_i(x_j)-\Exp[V_i(x_j)]$ we obtain
\begin{equation}\label{eq:ub-5}
    \max_{1\leq j \leq M} \Big\vert \sum_{i=1}^T V_i(x_j) \Big\vert^2 
    \le C
     \Big(\max_{1\leq j \leq M} \Big\vert \sum_{i=1}^T \zeta_{i,j}  \Big\vert^2 + (Th^d)^2\Big).
\end{equation}
For each fixed $j$, $\zeta_{i,j}$'s are zero mean independent random variables bounded in absolute value by $2K_{\max}$ and having variance bounded by $h^d \int K^2(u)du$ due to \eqref{eq:ub-4}. Thus, by Bernstein's inequality (see its generalized version in Lemma~\ref{lemma:lemma_8_rhode_tsybakov}) we have
\begin{equation*}
\mathbb{P}\goodpar{\left\vert\sum_{i=1}^T\zeta_{i,j}\right\vert \geq x} \leq 2\exp\goodpar{-{\tilde c} \frac{x^2}{Th^d + x}}, \quad \forall \; x >0,
\end{equation*}
where ${\tilde c}>0$ is a constant. It follows that for the normalized sums ${\bar \eta}_j:=\frac{1}{\sqrt{Th^d}}\sum_{i=1}^T\zeta_{i,j}$ we have
\begin{equation}\label{eq:ub-7a}
\mathbb{P}(|{\bar \eta}_{j}|>x) \le 2 \exp(-{\tilde c}x^2/2)\mathbf{I}\{x\le \sqrt{Th^d}\} + 
2 \exp(-{\tilde c}x\sqrt{Th^d}/2)\mathbf{I}\{x>\sqrt{Th^d}\}    
\end{equation}
for all $x>0, j=1,\dots,M$. This inequality and the union bound imply that, for any $M\ge 2$ and a constant $C>0$ large enough, 
\begin{align}\label{eq:ub-6}
    \Exp\goodbrak{\max_{1\leq j \leq M}|{\bar \eta}_{j}|^2} &= \int_0^{\infty} \mathbb{P}\goodpar{\max_{1\leq j \leq M}|{\bar \eta}_{j}|^2  >t}dt \\
    \nonumber
    &\leq C\log(M) +  \int_{C\log(M)}^{\infty} \mathbb{P}(\max_{1\leq j \leq M}|{\bar \eta}_{j}|>\sqrt{t})dt\\
    \nonumber
    & \le C\log(M) + \int_{C\log(M)}^{\infty} \sum_{j=1}^M\mathbb{P}(|{\bar \eta}_{j}|>\sqrt{t})dt\\
    \nonumber
    &\leq C\log(M) +  2M \int_{C\log(M)}^{x_0}\exp\goodpar{-{\tilde c}t/2}dt + 2M \int_{x_0}^{\infty}\exp\goodpar{-{\tilde c}\sqrt{tTh^d}/2}dt\\
    \nonumber
    &\leq C\log(M) +  2M \int_{C\log(M)}^{\infty}\exp\goodpar{-{\tilde c}t/2}dt + 2M \int_{Th^d}^{\infty}\exp\goodpar{-{\tilde c}\sqrt{tTh^d}/2}dt,
\end{align}
where $x_0=\max(C\log(M),Th^d)$.
By choosing here $C>0$ large enough we obtain that the first integral in the last line is bounded by a constant independent of $M$. Next,  with $c_0 = \tilde c/ 2$ and $y=Th^d$, since
$$\int_y^\infty \exp({-c_0\sqrt{yt}})\,dt
= \frac{2}{c_0^2 y}\int_{c_0y}^\infty u \exp({-u})\,du
= \frac{2}{c_0^2 y}(c_0y+1)\exp({-c_0y})
\leq c \exp({-c_0y}),
$$
for some constant $c>0$, and recalling that $M\le (2T^b)^d$, $h = a\goodpar{\frac{\log(T)}{T}}^{\fracbeta{1}}$, for the second integral we have
$$
M\int_{Th^d}^{\infty}\exp\goodpar{-{\tilde c}\sqrt{tTh^d}/2}dt \le  cM \exp(-Th^d c_0 )
\le c (2T^b)^d \exp\big(-c_0a^dT^{2\beta/(2\beta+d)}\big),
$$
which is also bounded by a constant independent of $M$.  
Putting these remarks together we conclude that there exists a constant $C'>0$ such that
\begin{equation}\label{eq:ub-7}
\Exp\goodbrak{\max_{1\leq j \leq M} \vert{\bar \eta}_j \vert^2} \leq C' \log(M).
\end{equation}
Combining \eqref{eq:ub-3}, \eqref{eq:ub-5} and \eqref{eq:ub-7} we obtain
\begin{align}
\label{eq:ub-8}
\Exp\goodbrak{\sup_{x \in [0,1]^d} \Big\vert\sum_{i=1}^T V_i(x) \Big\vert^2}
     & \le C \Big( \log(M) Th^d + (Th^d)^2+
     h^{-2}\epsilon^2 T^2\Big)
\end{align} 
for some constant $C>0$. Here, $M\le (2T^b)^d$, $h = a\goodpar{\frac{\log(T)}{T}}^{\fracbeta{1}}$. Choosing $b =\fracbeta{\beta+d+1}$ (so that $\epsilon = T^{-\fracbeta{\beta+d+1}}$) we obtain that the expression in \eqref{eq:ub-8} does not exceed $(Th^d)^2$ to within a constant factor. Using this fact in \eqref{eq:ub-2} we obtain that there exists $C^{+}>0$ such that
\begin{equation}\label{bias-final}
   \Exp\goodbrak{\supnorm{ \tilde{\Exp}[{\tilde f}^{(\alpha)}_T] - \derivativealpha }^2\inderound} \leq C^+h^{2\beta-2k}.
\end{equation}

{\it Control of the variance.} 
Consider now bounding from above the variance term on the event $\mathcal{E}$. Since 
\[
\tilde{\Exp} [{\tilde f}^{(\alpha)}_T (x)] = \sum_{i=1}^T f(X_i)\weights
\]
we have
\[
\Exp \goodbrak{\left\Vert{\tilde f}^{(\alpha)}_T - \tilde{\Exp}[{\tilde f}^{(\alpha)}_T (x)]\right\Vert^2_\infty \inderound} = \Exp\goodbrak{\sup_{x\in[0,1]^d}\left\vert\sum_{i=1}^T \xi_i \weights\right\vert^2 \inderound}.
\]
For $ i \in \{1,\dots,T\}, x \in [0,1]^d$, introduce the notation 
\[
S_i(x) := U\goodpar{\frac{X_i - x}{h}}K\goodpar{\frac{X_i-x}{h}}\in\real^{\cardinality{\ell}{d}}.
\]
Then we have
\begin{align*}
    \left|\sum_{i=1}^T \xi_i \weights \right| &= \left| \sum_{i=1}^T \xi_i \frac{1}{Th^{d+k}}U^{(\alpha)}(0)^{\top} \roundbtx^{-1} S_i(x) \right|
    \\
    &\leq \left\Vert \roundbtx^{-1}\sum_{i=1}^T \frac{1}{Th^{d+k}}\xi_i S_i(x)\right\Vert_2 
    \\
    &\leq \frac{1}{\lambda_0Th^{d+k}}\left\Vert \sum_{i=1}^T \xi_i S_i(x) \right\Vert_2, \notag 
\end{align*}
where we used the Cauchy-Schwarz inequality and the fact that $\Vert U^{(\alpha)}(0)\Vert_2 = 1$. 
It follows that
\[
A := \sup_{x\in [0,1]^d} \left|\sum_{i=1}^T \xi_i \weights \right| \inderound \leq \frac{1}{\lambda_0Th^{d+k}} \sup_{x\in[0,1]^d}\left\Vert \sum_{i=1}^T \xi_i S_i(x) \right\Vert_2.\notag
\]
 To control the last supremum we use again an $\epsilon$-net argument. We consider the $\epsilon$-net composed of the same points $x_1,\dots,x_M$ as defined above. We have
\begin{equation}\label{ineq:first_decomp_a_square}
    A^2 \leq \goodpar{\frac{1}{\lambda_0 Th^{d+k}}}^2 \left(\max_{1\leq j \leq M} \left\Vert \sum_{i=1}^T \xi_i S_i(x_j) \right\Vert_2  + \sup_{x,x' : \Vert x-x'\Vert_{\infty} \leq \epsilon} \left\Vert \sum_{i=1}^T \xi_i (S_i(x')- S_i(x))  \right\Vert_2\right)^2.
\end{equation} 
We first evaluate the second term on the right hand side of this inequality. For any $x,x' \in [0,1]^d$, 
\[
    \sum_{i=1}^T \xi_i(S_i(x)-S_i(x')) = \sum_{i=1}^T
    \xi_i \goodpar{U\goodpar{\frac{X_i-x}{h}}K\goodpar{\frac{X_i-x}{h}} - U\goodpar{\frac{X_i-x'}{h}}K\goodpar{\frac{X_i-x'}{h}}}.
\]
We will use the following lemma. 
\begin{Lemma}\label{lem:lip_ku}
    There exists a constant $\overline{L}$ depending only on $K_{max}, L_K, d, \ell$ such that $U(\cdot)K(\cdot)$ is $\overline{L}$-Lipschitz, i.e., $\forall \; u,u' \in \real^d$ : 
    \[
    \euclideannorm{ U(u)K(u) - U(u')K(u')}\leq \overline{L}\euclideannorm{u-u'}
    \]
\end{Lemma} 
\begin{proof}
To prove the lemma, we establish a Lipschitz condition on $U(\cdot)K(\cdot)$ coordinate-wise. We show that $$\forall \; u, v \in \real^d, \; \forall \; j \in \{1,\dots,\cardinality{\ell}{d}\}, \; |U_j(u)K(u) - U_j(v)K(v)|\leq L_{UK_j} \euclideannorm{u-v},$$ where $ L_{UK_j}$ is a Lipschitz constant. Then, the constant $\overline{L}$ in the lemma is the sum of the squared Lipschitz constants of each coordinate. To show the relation coordinate-wise, we consider two cases. In the first case, both $u$ and $v$ belong to the support of $K$ (denoted Supp($K$)). In the second, at least one of $u$ or $v$ do not belong to Supp($K$). For $u,v\in\text{Supp}(K)$, we start by noticing that $|U_j(u)|\leq U_{\text{max}}<\infty$ by the definition of $U(\cdot)$ and the fact that we consider points on a compact set of $\real^d$. We also have that 
\[
|U_j(u)K(u) - U_j(v)K(v)|=\left|U_j(u)\goodpar{K(u)-K(v)} - K(v)(U_j(v)-U_j(u))\right|. 
\]
As $u,v\in\text{Supp}(K)$, $U_j$ is Lipschitz as a polynomial on a compact support. Let $L_{U_j}$ denote its Lipschitz constant. Also, $K$ is $L_K$-Lipschitz. Using the triangle inequality, and the uniform bounds on $U_j$ and $K$ we get the following upper bound :
\[
|U_j(u)K(u) - U_j(v)K(v)| \leq (K_{\text{max}}L_{U_j}+U_{\text{max}}L_K)\euclideannorm{u-v}.
\]
Now, for the second case, at least one of $u$ or $v$ is outside $\text{Supp}(K)$. If none of them are in the support of $K$, then $U_j(u)K(u) - U_j(v)K(v) = 0$. If only one of them (let it be $u$) is in the support of $K$, we have 
\[
|U_j(u)K(u)-U_j(v)K(v)| = |U_j(u)K(u)|= |U_j(u)(K(u)-K(v))| \leq U_{\text{max}}L_K\euclideannorm{u-v}.
\]
Hence, for any coordinate $j$, we have that $(U(\cdot)K(\cdot))_j$ is at most $(K_{\text{max}}L_{U_j}+U_{\text{max}}L_K)$-Lipschitz.
It follows that $U(\cdot)K(\cdot)$ is $\overline{L}$-Lipschitz with $\overline{L} =( \sum_{j=1}^{\cardinality{\ell}{d}}(K_{\text{max}}L_{U_j}+U_{\text{max}}L_K)^2)^{1/2}$.  
\end{proof}
Applying Lemma \ref{lem:lip_ku} yields, with $\overline{L}'=\overline{L}^2 d$,
\begin{align}
\goodpar{\frac{1}{\lambda_0Th^{d+k}}}^2 \sup_{x,x' : \Vert x-x'\Vert_{\infty} \leq \epsilon} \left\Vert \sum_{i=1}^T \xi_i (S_i(x')- S_i(x))  \right\Vert_2^2 
& \leq \frac{\epsilon^2\overline{L}'}{\lambda_0^2 T^2 h^{2d+2k+2}}\goodpar{\sum_{i=1}^T|\xi_i|}^2.
\label{ineq:bound_on_second_term_regularity_variance}
\end{align}
Taking expectations in (\ref{ineq:first_decomp_a_square}), using the inequality $(a+b)^2\leq2a^2 + 2b^2$, \eqref{ineq:bound_on_second_term_regularity_variance} and Jensen's inequality we obtain:
\begin{equation}\label{eq:def_eta_j}
    \Exp[A^2] \leq \frac{2}{\lambda_0^2 T h^{d+2k}} \Exp\goodbrak{\max_{1\leq j \leq M} \Vert \eta_j \Vert^2} + \frac{2\epsilon^2\overline{L}'\sigma_{\xi}^2}{\lambda_0^2h^{2d+2k+2}}
\end{equation}
with the random vectors
\begin{equation*}
    \eta_j := \frac{1}{\sqrt{Th^d}}\sum_{i=1}^T \xi_i S_i(x_j).
\end{equation*}
It remains now to control $\Exp[\max_{1\leq j\leq M} \Vert\eta_j\Vert_2^2]$. To this end, we use the following vector version of Bernstein's inequality (see \cite{pinelis1986largedev} or Lemma 8 in \cite{rohde2011estimation}).
\begin{Lemma}\label{lemma:lemma_8_rhode_tsybakov}
Let $\zeta_1,\dots,\zeta_T$ be independent zero mean random vectors such that  
\begin{equation}\label{eq:condition_lemma_rhode_tsyb}
    \sum_{i=1}^T \Exp\goodbrak{\Vert \zeta_i \Vert_2^l} \leq \frac{1}{2}l! B^2 H^{l-2}, \quad l=2,3,\dots
\end{equation}
with some finite constants $B,H>0$. Then, 
\begin{equation*}
\mathbb{P}\goodpar{\left\Vert\sum_{i=1}^T\zeta_i\right\Vert_2 \geq x} \leq 2\exp\goodpar{-\frac{x^2}{2B^2 + 2xH}}, \quad \forall \; x >0.
\end{equation*}
\end{Lemma}
We use this lemma with $\zeta_i= \zeta_i(j):= \frac{1}{\sqrt{Th^d}}\xi_iS_i(x_j)$. Note that $\zeta_i$'s are zero mean random vectors. We now check that assumption \eqref{eq:condition_lemma_rhode_tsyb} is satisfied and specify the corresponding values of $B$ and $H$.  Recall that $K(\cdot)$ is compactly supported and all components of $U(\cdot)$ are uniformly bounded on ${\rm Supp}(K)$, so that $\|U(u)\|_2\le U_*<\infty$ for all $u\in {\rm Supp}(K)$. Thus for any $l\geq2$,
\begin{align*}
    \Exp\goodbrak{\|S_i(x_j)\|_2^l} &= \Exp\goodbrak{\Big\Vert U\goodpar{\frac{X_i - x_j}{h}}\Big\Vert_2^l K\goodpar{\frac{X_i - x_j}{h}}^l}\leq (U_*K_{\max})^l\Exp\goodbrak{\mathbf{I}\left\{\left\Vert \frac{X_i - x_j}{h}\right\Vert_2\leq1\right\}} \\
    & \leq C(d)(U_*K_{\max})^l h^d,
\end{align*}
where $C(d)$ is the volume of $\mathbb{B}_d$, which is smaller than 6 for any $d\geq1$. 
Next, since $\xi_i$'s are conditionally sub-Gaussian we have, for all $l\ge 2$,
\begin{equation*}
    \Exp\goodbrak{|\xi_i|^l\vert X_i} \leq 2^{l}\sqrt{l!} \sigma_{\xi}^{l}
\end{equation*}
(see, e.g., Theorem 2.1 in \cite{boucheron2013concentration}). Summing over $i$ and using the last two displays yields:
\[
\sum_{i=1}^T\Exp\goodbrak{\|\xi_i S_i(x_j)\|_2^l} = \sum_{i=1}^T\Exp\goodbrak{\| S_i(x_j)\|_2^l\Exp\goodbrak{\left\vert\xi_i\right\vert^l \vert X_i}}\leq
C(d)(2 U_*K_{\max}\sigma_{\xi})^l \sqrt{l!}  \, Th^d.
\]It follows that, for all $l\ge2$,
\[ 
\sum_{i=1}^T \Exp\goodbrak{\Big\Vert \frac{1}{\sqrt{Th^d}}\xi_i S_i(x_j)\Big\Vert_2^l} \leq
\frac{l!}{2}B^2 H^{l-2}
\]
with $B^2 = 8C(d)(U_*K_{\max}\sigma_{\xi})^2$,  $H=\frac{2U_*K_{\max}\sigma_{\xi}}{\sqrt{Th^d}}$.  Recalling that $\eta_j= \sum_{i=1}^T\zeta_i(j)$ and applying Lemma~\ref{lemma:lemma_8_rhode_tsybakov} with these values of $B$ and $H$ we obtain that, for some constants $c_1,c_2>0$,
$$
\mathbb{P}(\|\eta_{j}\|_2>x) \le 2 \exp(-c_1x^2)\mathbf{I}\{x\le \sqrt{Th^d}\} + 
2 \exp(-c_2x\sqrt{Th^d})\mathbf{I}\{x>\sqrt{Th^d}\}, \quad \forall x>0, j=1,\dots,M.
$$
This inequality is analogous to \eqref{eq:ub-7a}. Therefore, using the argument as in \eqref{eq:ub-6} and further up to \eqref{eq:ub-7} with $\|\eta_{j}\|_2$ instead of $|\bar \eta_{j}|$ we conclude that
 there exists a constant $C'>0$ such that
\begin{equation}\label{ineq:bound_squared_norm_eta_j}
\Exp\goodbrak{\max_{1\leq j \leq M} \Vert\eta_j \Vert_2^2} \leq C' \log(M).
\end{equation}
Plugging  (\ref{ineq:bound_squared_norm_eta_j}) in the bound (\ref{eq:def_eta_j}) yields: 
\[
 \Exp \goodbrak{\supnorm{{\tilde f}^{(\alpha)}_T - \tilde{\Exp}[{\tilde f}^{(\alpha)}_T]}^2\inderound} \leq \frac{C''\log(M)}{Th^{d+2k}} + \frac{2\epsilon^2\overline{L}'\sigma_{\xi}^2}{\lambda_0^2h^{2d+2k+2}}.
\]
Here, $M\le (2T^b)^d$, $h = a\goodpar{\frac{\log(T)}{T}}^{\fracbeta{1}}$ and $b$ is chosen as $b =\fracbeta{\beta+d+1}$ (so that $\epsilon = T^{-\fracbeta{\beta+d+1}}$). Thus, there exists a constant $C_{1}>0$ such that
\[
\Exp \goodbrak{\supnorm{{\tilde f}^{(\alpha)}_T - \tilde{\Exp}[{\tilde f}^{(\alpha)}_T]}^2\inderound} \leq C_{1} \goodpar{\frac{\log(T)}{T}}^{\fracbeta{2\beta-2k}}.
\]
Combining this inequality with \eqref{ineq:ineq_first_decomp_sup_norm_bound_proba}, \eqref{bias-variance-decomp}, and \eqref{bias-final} proves the theorem.

\section{Proof of the lower bound (Theorem \ref{th:lower_bound})}\label{app:B}

Recall that $\lfloor \beta \rfloor=\ell$, so functions $f\in \Sigma(\beta,L)$ admit partial derivatives $f^{(s)}$ up to order $\ell$, and all derivatives $f^{(s)}$ with $\|s\|_1=\ell$ satisfy a Hölder condition. It suffices to restrict attention to a suitably chosen subclass of $\Sigma(\beta,L)$. The proof of Theorem~\ref{th:lower_bound} relies on the use of a version of Fano's inequality and taking care of the observations dependence as we consider sequential strategies.

\paragraph{Construction of the family of functions}
 Decompose $[0,1]^d$ into a partition of cubes $\{A_j\}_{j=1}^N$. The $A_j$ are disjoint cubes of side length $1/K_1$, with centers $\mu_j$, and we set 
\[
K_1 := \left( \frac{T}{\log(T)}\right)^{\fracbeta{1}}.
\]
For notational convenience and without loss of generality, we assume that $N = K_1^d$ is an integer (replacing $K_1^d$ by either $\lfloor K_1^d \rfloor$ or $\lceil K_1^d \rceil$ does not affect the resulting rates). The assumption 
$T~\ge~5^{(2\beta/d)+1}$ grants that $N\ge 3$.
To each cube $A_j$, $j=1,\dots,N$, we associate a function $f_j$, and we set $f_1(\cdot)=0$. Each $f_j$ is supported on $A_j$. To define $f_j$, we introduce, for each $j$, 
\[
\forall u \in [0,1]^d, \quad G_j(u) := \prod_{i=1}^d g_i\bigl(K_1(u_i - \mu_{j,i})\bigr),
\]
where $u_i$ denotes the $i$-th coordinate of $u$ and $\mu_{j,i}$ the $i$-th coordinate of $\mu_j$. For each $i \in \{1,\dots,d\}$,  $g_i(\cdot)$ is an infinitely many times differentiable function supported on $(-1/2,1/2)$, with $\sup_{u} |g_i^{(s)}(u)| \leq g_{\max} < \infty$ for all integers $0 \leq s \le \ell$.
We rescale $g_i$'s in such a way that $\max_{u\in (-1/2,1/2)^d } \prod_{i=1}^d g_i^{(\alpha_i)}(u_i)=1$.

Define functions $f_j$, $j=1,\dots, N$, as follows:
\begin{equation*}
    f_j(u) =
    \begin{cases}
        -a\, G_j(u), & \text{if } u \in A_j, \\
        0, & \text{otherwise},
    \end{cases}
\end{equation*}
where
\[
a := a_0 \left( \frac{T}{\log T} \right)^{-\fracbeta{\beta}} = a_0 K_1^{-\beta}.
\]
By choosing $a_0$ sufficiently small, we can ensure that all functions $f_j$ belong to $\Sigma(\beta,L)$.  Moreover, 
\begin{equation*}
\min_{x\in[0,1]^d} f^{(\alpha)}_j(x)=\min_{x\in A_j} f^{(\alpha)}_j(x) =-a K_1^{k}\max_{u\in (-1/2,1/2)^d } \prod_{i=1}^d g_i^{(\alpha_i)}(u_i)  = -a_0(\log(T)/T)^{\fracbeta{\beta-k}} =: -a_{*}. 
\end{equation*}
We now fix an arbitrary sequence of estimators $(\hat{x}_i)_{i=1}^T$ and define 
\[
R_T := \sup_{f\in\betahold} \Exp_{f}\left[ \sum_{i=1}^T\goodpar{f\exponentalpha(\hat{x}_i) - \min_{x\in[0,1]^d}f\exponentalpha(x)} \right]
\]
and the events
\[
\randomev  = \{\hat{x}_i \not \in A_j
\}, \quad i=1,\dots,T, \ \ j=1,\dots, N.
\]
Our aim is to bound $R_T$ from below uniformly over  $(\hat{x}_i)_{i=1}^T$. Conditioning on $\tau$ we have
\begin{align}\label{eq:start-lb}
    \quad R_T & \geq \sup_{j=1,...,N} \sum_{i=1}^T\Exp_{f_j}\left[ f_j\exponentalpha(\hat{x}_i) - \min_{x\in[0,1]^d} f_j\exponentalpha(x)\right]  \\ 
    &\ge \frac{1}{N}\sum_{j=1}^N\sum_{i=1}^T \Exp_{f_j}\left\{\Exp_{f_j}\left[  f_j\exponentalpha(\hat{x}_i) - \min_{x\in[0,1]^d} f_j\exponentalpha(x) \Big\vert \tau \right]\right\} \nonumber
\end{align}
In what follows, we obtain a lower bound independent of $\tau$ for the expression
\begin{align}\label{eq:lb expression}
\frac{1}{N}\sum_{j=1}^N\sum_{i=1}^T \Exp_{f_j}\left[  f_j\exponentalpha(\hat{x}_i) - \min_{x\in[0,1]^d} f_j\exponentalpha(x) \Big\vert \tau \right],
\end{align}
which implies the same bound for $R_T$. In order to bound \eqref{eq:lb expression} for fixed $\tau$ it is enough to bound the same expression, where the expectations are unconditional while  $X_i=\bar\Phi_i((X_t,Y_t)_{t=1}^{i-1})$ and $\hat{x}_i = \bar\Psi_i((Y_t,X_t)_{t=1}^{i})$ 
with  arbitrary measurable functions $\bar\Phi_i$ and $\bar\Psi_i$. It means that at this stage we can act as if there were no randomization $\tau$. Therefore, to shorten the notation, in the subsequent argument we write
$\Exp_{f_j}\left[\cdot\right]$ instead of $\Exp_{f_j}\left[\cdot\vert \tau\right]$.
With this convention and using the fact that $f_j\exponentalpha(\hat{x}_i)$ on the event $\randomev$, the expression in \eqref{eq:lb expression} can be bounded from below as follows:
\begin{align*}
    &\frac{1}{N}\sum_{j=1}^N\sum_{i=1}^T \Exp_{f_j}\left[  f_j\exponentalpha(\hat{x}_i) - \min_{x\in[0,1]^d} f_j\exponentalpha(x)\right]  \\ 
    &\quad \ge \frac{1}{N}\sum_{j=1}^N\sum_{i=1}^T \Exp_{f_j}\left[ \left( f_j\exponentalpha(\hat{x}_i) - \min_{x\in[0,1]^d} f_j\exponentalpha(x)\right)\mathbf{I}\{\randomev\}\right]\notag \\
    & \quad = \frac{1}{N}\sum_{i=1}^T\sum_{j=1}^N\Exp_{f_j}\left[ -\min_{x\in[0,1]^d} f_j\exponentalpha(x)\, \mathbf{I}\{\randomev\}\right]\notag = a_{*}\sum_{i=1}^T\left[ \frac{1}{N} \sum_{j=1}^N \proba_{f_j}(\randomev)\right].\label{ineq:lower_bound_cumureg_error_term}
\end{align*}
Here, $\proba_{f_j}$ denotes the distribution of the data when the true function is $f_j$. 

Next, we show that the value
$$
p_T:= \min_{i=1,\dots,T} \frac{1}{N}\sum_{j=1}^N \proba_{f_j}(\randomev)
$$
is bounded from below by an absolute constant. This will complete the proof.

We relate $p_T$ to the average probability of error of a multiple testing problem and then bound it from below by the minimal error over all tests. To this end, we first define a specific test $S^*$ measurable with respect to the estimator $\hat{x}_i$:
\[
S^*(\hat{x}_i) = \sum_{j=1}^N j \mathbf{I}\{\hat{x}_i \in A_j\}. 
\]
Note that, for any $i\in\{1,\dots,T\}, \; j\in\{1\dots,N\}$, we have $S^*(\hat{x}_i) = j$ if and only if $\randomev^c$ holds. Since the supports of functions $f_j$  are disjoint the events $\randomev^c,\; j=1,...,N,$ are also disjoint for any fixed $i\in\{1,\dots,T\}$. 
For any $i\in\{1,\dots,T\}$, we have
\begin{align*}
   \frac{1}{N}\sum_{j=1}^N \proba_{f_j}(\randomev) &= \frac{1}{N}\sum_{j=1}^N \proba_{f_j}(S^*(\hat{x}_i) \neq j) \\
   &\ge  \inf_S \frac{1}{N}\sum_{j=1}^N\proba_{f_j}(S\neq j),
\end{align*}
where the infimum is taken over all statistics $S$ with values in $\{1\dots,N\}$. The last expression is bounded from below using the version of Fano Lemma given in Lemma~\ref{cor:cor_stybakov_2009}.
For two probability measures $P$ and $Q$, consider the Kullback-Leibler divergence defined as $\KL(P,Q): = \int\log(dP/dQ)dP$ if $P\ll Q$ and $\KL(P,Q):=+\infty$ otherwise. 
\begin{Lemma}\label{cor:cor_stybakov_2009}[\cite{Tsybakov09}, Corollary 2.6]
Let $P_0,P_1,\dots,P_M$ be probability measures and $M\geq2$. Let \[
\frac{1}{M+1}\sum_{j=1}^M {\rm KL}(P_j,P_0) \leq \alpha \log (M),
\]
where $0<\alpha<1$. 
Then, 
\[
\inf_S \frac{1}{M+1}\sum_{j=0}^MP_j(S\neq j) \geq \frac{\log(M+1) - \log(2)}{\log(M)}-\alpha,
\]
where $\inf_S$ denotes the infimum over all statistics with values in $\{0,1,...,M\}$.
\end{Lemma}
In our case, we set $P_0 = \proba_{f_1}, \; P_1 = \proba_{f_2},...,\;P_M=\proba_{f_N}$, and hence $M=N-1$. Recall that $N \geq 3$. For $M\geq 2,$ we have $\frac{\log(M+1) - \log(2)}{\log(M)}\geq\frac{1}{4}$. We fix $\alpha = \frac{1}{8}$. If
\begin{equation}\label{eq:ineq_kl_for_corollary}
\frac{1}{N}\sum_{j=2}^N \KL(\proba_{f_j}, \proba_{f_1}) \leq \frac{\log(N-1)}{8},
\end{equation}
then Lemma~\ref{cor:cor_stybakov_2009} yields
\begin{equation}\label{eq:ineq_lb for prob of error}
    \inf_S \frac{1}{N} \sum_{j=1}^N \proba_{f_j}(S \neq j) \geq \frac{1}{8}.
\end{equation}
We prove at the end of this section that \eqref{eq:ineq_kl_for_corollary} holds if $a_0$ is chosen small enough.
Combining \eqref{eq:ineq_lb for prob of error} with the preceding argument leads to the bound
\begin{equation}\label{eq:ineq_lb for cum regret}
    R_T \geq  \frac{a_{*}T}{8} = \frac{a_0}{8} T^{\fracbeta{\beta + d+k}}\log(T)^{\fracbeta{\beta-k}}.
\end{equation}
Since inequality \eqref{eq:ineq_lb for cum regret} holds for any sequence of estimators $(\hat{x}_i)_{i=1}^T$ obtained through active procedures the desired minimax lower bound on the cumulative regret follows. The lower bound on the simple regret is obtained quite analogously by following the steps after \eqref{eq:start-lb}, where the sums $\sum_{i=1}^T$  are replaced by only one element in the sum corresponding to index $i=T$. This leads to the lower bound ${a_{*}}/{8}$ on the minimax simple regret.

\paragraph{Proof of \eqref{eq:ineq_kl_for_corollary}} To conclude the proof, we now show that (\ref{eq:ineq_kl_for_corollary}) holds if $a_0$ is chosen small enough. Note that if $\proba_{f_j}\ll \proba_{f_1}$ then, by the chain rule for sequential strategies, see, e.g.,\cite[Theorem 2.16]{polyanskiy2025information} we have
$$
\frac{d\proba_{f_j}}{d\proba_{f_1}}\big((X_i,Y_i)_{i=1}^T\big)= \prod_{i=1}^T \frac{dF_{\xi}(Y_i - f_j(X_i))}{dF_{\xi}(Y_i - f_1(X_i))}.
$$
Using Assumption 4 and the fact that the supports of $f_1$ and $f_j$ are disjoint we obtain that the Kullback-Leibler divergence $\KL(\proba_{f_j},\proba_{f_1})$ satisfies:
\begin{align*}
    \KL(\proba_{f_j},\proba_{f_1})& =  \Exp_{f_j}\left[\sum_{i=1}^T\log\left(\frac{dF_{\xi}(Y_i-f_j(X_i))}{dF_{\xi}(Y_i-f_1(X_i))}\right)\right] \\
    &= \Exp_{f_j}\goodbrak{\sum_{i=1}^T \int \log\goodpar{\frac{dF_{\xi}(t)}{dF_{\xi}(t + f_j(X_i))}} dF_{\xi}(t) }\\
    &\le \Exp_{f_j}\left[\sum_{i=1}^T I_0 f_j^2(X_i)\right],
\end{align*}
where the last inequality holds whenever $\sup_u \vert f_j (u) \vert \le v_0$, which is granted if we choose $a_0\le v_0 g_{\max}^{-d}$ (recall that $\sup_u \vert f_j (u) \vert \le a_0 g_{\max}^{d}\left( \frac{T}{\log T} \right)^{-\fracbeta{\beta}}$ by construction). 
Thus, for such $a_0$ we have
\[
    \frac{1}{N}\sum_{j=2}^N {\rm KL}(\proba_{f_j},\proba_{f_1}) \leq \frac{I_0}{N}\sum_{i=1}^T\sum_{j=2}^N\Exp_{f_j}\left[f_j^2(X_i)\right]. 
\]
Moreover, since the supports of $f_j$'s are disjoint, 
\[
\sum_{j=2}^Nf_j^2(X_i) \leq \max_{x,j}f_j^2(x) = a^2.
\]
Therefore, 
\begin{equation*}
    \frac{1}{N}\sum_{j=2}^N\KL(P_{f_j},P_{f_1}) \leq \frac{I_0a^2}{N}T  
    = I_0 a_0^2 K_1^{-d-2\beta} T 
    = I_0a_0^2\log(T).
\end{equation*}
On the other hand, for $N\geq3$,
\begin{equation*}
    \frac{\log(N-1)}{8} \geq \frac{\log(N/2)}{8}= \frac{\log(K_1^d/2)}{8}\geq c_0\log(T),
\end{equation*}
where $c_0 >0$ is some constant. We conclude that  (\ref{eq:ineq_kl_for_corollary}) holds if $a_0\le \min(\sqrt{c_0/I_0},v_0 g_{\max}^{-d})$. \\
\hspace*{\fill} $\square$

\section{Proof of Theorem \ref{prop:rate_for_feasible_estim} (the polynomial time estimator)}\label{app:C}

Recall that, at each time instance $j$, the learner has access to $j$ observations pairs $(X_t,Y_t)_{t=1}^j$ given by the model, and an independent sample $Z_1,\dots,Z_{V_j}$ drawn from the uniform distribution on $[0,1]^d$.  Let $x^* \in [0,1]^d$ be a minimizer of $f\exponentalpha$
on $[0,1]^d$.  By \eqref{eq:def_of_feasible_estim}, we have
\begin{equation}\label{eq:2terms}
  f\exponentalpha(\hat{z}_{V_j})-f\exponentalpha(x^*) \leq 2\sup_{x\in[0,1]^d}\vert \estimfjalpha(x)-f\exponentalpha(x)\vert + {f\exponentalpha(z_{V_j}^*) - f\exponentalpha(x^*)}.  
\end{equation}
It is not hard to check that $ \lambda (  B ( x^* ,  \varepsilon)  \cap [0,1]^d ) \geq  c_d \lambda (  B ( x^* ,  \varepsilon)  )  ) $, for any $\varepsilon \in (0,1/2] $, where $c_d>0$ depends only on $d$. 
Let $\mathcal{Y}$  be the event that at least one $Z_i$ falls $\varepsilon$-close to $x^*$, i.e., $\mathcal{Y}:=\bigcup_{i=1}^{V_j} \{\Vert Z_i - x^*\Vert_2\leq\varepsilon\}$.
For any $\varepsilon \in (0,1/2] $ we have
$$ \proba\goodpar{\mathcal{Y}^c} = \goodpar{1 -   \lambda (  B ( x^* ,  \varepsilon)   \cap [0,1]^d) }^{V_j}  \leq   \goodpar{1 -    c_d \lambda (  B ( x^* ,  \varepsilon)  )  }^{V_j}  = \goodpar{1- c_dC_d  \varepsilon^d} ^{V_j}.$$ 
where $C_d := \pi^{d/2}/\Gamma(\frac{d}{2}+1)$.
Define $\varepsilon = (\log(V_j)/(V_jc_dC_d))^{\frac{1}{d}} \wedge 1/2$. Then
\begin{eqnarray}\label{eq:proba_zi_close_x*}
    \proba\goodpar{\mathcal{Y}^c} \leq  \frac {1} {V_j}  \vee \exp(-c_dC_d 2^{-d}V_j)\leq  \frac {M_d } {V_j}   ,
\end{eqnarray}
for some $M_d>0$ depending only on $d$.
We have
\begin{align*}
   \mathbb E [  {f\exponentalpha(z_{V_j}^*) - f\exponentalpha(x^*)}] & = 
  \Exp\goodbrak{(f\exponentalpha(z_{V_j}^{*})-f\exponentalpha(x^*))\mathbf{I}\{\mathcal{Y}\}} 
    +  \Exp\goodbrak{(f\exponentalpha(z_{V_j}^{*})-f\exponentalpha(x^*))\mathbf{I}\{\mathcal{Y}^c\}}\\
    & \leq  \Exp\goodbrak{f\exponentalpha(z_{V_j}^{**})-f\exponentalpha(x^*)\mathbf{I}\{\mathcal{Y}\}} + \frac{2LM_d}{V_j},
\end{align*}
 where $z_{V_j}^{**} \in\{Z_1,\dots,Z_{V_j}\}$ is any discretization point satisfying $\|z_{V_j}^{**}-x^*\|_2\leq \varepsilon$. Using the H\"older property of $f$ and considering different values of $\lonenorm{\alpha}=k$, we have : 
\begin{eqnarray}\label{ineq:bound_above_a_1}
    \Exp\goodbrak{(f\exponentalpha(z_{V_j}^{**})-f\exponentalpha(x^*))\mathbf{I}\{\mathcal{Y}\}} \leq \begin{cases}
        L \varepsilon^{\beta-\ell} \; \text{if} \; \lonenorm{\alpha} = \ell, \\
       \bar C \varepsilon \quad \text{if} \; \lonenorm{\alpha}< \ell, 
    \end{cases}
\end{eqnarray}
where $\ell=\lfloor \beta\rfloor$.
 The bound for the first case $\lonenorm{\alpha} = \ell$ in \eqref{ineq:bound_above_a_1} follows directly from the definition of $\beta$-Hölder functions. The bound for the second case uses the fact that for $\lonenorm{\alpha} \le \ell-1$ all the directional derivatives of $f\exponentalpha$ are bounded by $L$, so that the result holds with $\bar C = \sqrt d L$.  
 
 Consider now separately the two  cases of (\ref{ineq:bound_above_a_1}). For the case $\lonenorm{\alpha}< \ell$, the choice $V_j = \lfloor j^{\gamma}\rfloor +1$ with $\gamma / d>   \frac{(\beta-k)}{2\beta+d}$ implies that
$$  
\Exp[f\exponentalpha(z_{V_j}^{*})-f\exponentalpha(x^*)] \leq  \bar C  \left( \frac{\log(V_j)}{V_jc_dC_d} \right)^{\frac{1}{d}} +  \frac{2LM_d}{V_j} \leq C (\log(j))^{\frac{1}{d}} j^{- \frac{(\beta-k)}{2\beta+d}-\eta},$$
for some $C>0$ and $\eta>0$. For the case $\lonenorm{\alpha} = \ell$, using the fact that $V_j = \lfloor j^{\gamma}\rfloor +1$ with $\gamma  / d>   \frac{1}{2\beta+d}$ and recalling the notation $k=\lonenorm{\alpha}$ we obtain: 
$$  \Exp[f\exponentalpha(z_{V_j}^{*})-f\exponentalpha(x^*)] \leq   L \left( \frac{\log(V_j)}{ V_j c_d C_d} \right)^{\frac{\beta-\ell }{d}} +  \frac{2LM_d}{V_j} \leq C' (\log(j))^{\frac{\beta-k}{d}} j^{- \frac{(\beta-k)}{2\beta+d}-\eta},
$$
for some $C'>0$ and $\eta>0$. The two above upper bounds decay faster than $ (\log(j) / j)^{\fracbeta{\beta-k}} $. Together with \eqref{eq:2terms} and Proposition~\ref{prop:prop_upper_bound_lpe} this  implies the desired result. \\
\hspace*{\fill} $\square$

\bibliography{bib.bib}
\end{document}